\documentclass[reqno,a4paper,11pt]{amsart}

\usepackage{graphicx} % Required for inserting images
\usepackage{amsfonts}
\usepackage{amsmath}
\usepackage{amssymb}
\usepackage{amsthm}
\usepackage{esint}
\usepackage{geometry}
\usepackage{musicography}
\newtheorem{theorem}{Theorem}
\newtheorem{proposition}[theorem]{Proposition}
\newtheorem{definition}[theorem]{Definition}
\newtheorem{corollary}[theorem]{Corollary}
\newtheorem{lemma}[theorem]{Lemma}
\newtheorem{remark}[theorem]{Remark}
\usepackage{esint}
\usepackage{appendix}

\usepackage[hidelinks]{hyperref}

\newcommand{\dummy}{\,\bullet\,}

\newcommand{\R}{\mathbb{R}}
\newcommand{\C}{\mathbb{C}}
\newcommand{\N}{\mathbb{N}}
\newcommand{\K}{\mathbb{K}}
\newcommand{\mO}{\setminus\{0\}}

\DeclareMathOperator{\rad}{rad}
\newcommand{\generate}[1]{\left\langle#1\right\rangle}

\newcommand{\sphere}{\mathbb{S}}

\newcommand{\proj}{\mathrm{proj}}

\newcommand{\Spec}{\mathrm{Spec}}

\newcommand{\abs}[1]{\left\lvert #1 \right\rvert}

\newcommand{\fullref}[1]{(\ref{#1})}

\title{Horizontal and Vertical Regularity of Elastic Wave Geometry}
\author{Joonas Ilmavirta, Pieti Kirkkopelto and Antti Kykkänen}
\date{\today}

\begin{document}

\begin{abstract}
The elastic properties of a material are encoded in a stiffness tensor field and the propagation of elastic waves is modeled by the elastic wave equation. We characterize analytic and algebraic properties a general anisotropic stiffness tensor field has to satisfy in order for Finsler-geometric methods to be applicable in studying inverse problems related to imaging with elastic waves.
\end{abstract}

\maketitle

\section{Introduction}

The propagation of elastic waves is accurately described by geodesics in a certain Finsler geometry, and we study how the regularity of this geometry depends on the properties of the stiffness tensor field that describes the elastic material. There are two types of seismic waves, the qP- and qS-waves, and our main focus will be on the Finsler structure arising from the qP-waves, which we will refer to as the elastic wave geometry.

In two dimensions we give a complete characterization of stiffness tensors that give rise to an elastic wave geometry that is $k$-times continuously differentiable along the manifold and smooth on the fibers, and in dimensions greater than two we show that there is an open neighborhood of the isotropic stiffness tensors where the same regularity continues to hold. We also show that for complex stiffness tensors the slowness surface is always singular as a scheme in dimensions $d\notin\{2,4,8\}$. We apply the regularity results to two inverse problems; we prove the injectivity of the geodesic X-ray transform on scalar fields in some specific low regularity elastic geometries relevant for seismology and as a second application under suitable assumptions the determination of low regularity elastic wave geometries up to a Finslerian isometry from travel time data.

\subsection{Main Results}

The horizontal regularity of the elastic wave geometry associated to a stiffness tensor field $c$ depends on the regularity of $c$ along the manifold, while the vertical regularity depends on the algebraic properties of $c$.

\subsubsection{Horizontal Regularity}
Our main result concerning horizontal regularity is the following.

\begin{theorem}[{Proven in Section~\ref{sec:thm-1-proof}}]
\label{hregmain}
    Let $M\subset\R^n$ be a smooth domain with or without boundary and let $c\in C^k(M)$ be a stiffness tensor field such that the qP-branch of the slowness surface associated to $c$ is globally separate. Then the associated Finsler function $F^c_{qP}$ is of class $C^k$ along the manifold and smooth on the fibers.
\end{theorem}

In order to prove Theorem~\ref{hregmain} we introduce a class of anisotropic function spaces and prove anisotropic implicit and inverse function theorems on general smooth fiber bundles, a result we consider of interest in its own right. For details we refer to Section~\ref{sec:horizontal-regularity}.

\subsubsection{Vertical Regularity}

In light of Theorem~\ref{hregmain}, the vertical regularity of the elastic wave geometry depends on whether the $qP$-branch of the slowness surface is separate. 

In two dimensions we have a complete characterization of stiffness tensors with separate slowness surfaces:

\begin{theorem}[{Proven in Section~\ref{sec:singularity-2D}}]
\label{2dseparate}
    Let $c$ be a two-dimensional stiffness tensor. Then there exist explicit polynomials $R$ and $D$ in the components of $c$ such that the slowness surface of $c$ is real and separate if and only if
    \begin{equation}
        R(c)\neq 0\quad \text{and}\quad D(c)\ge 0.
    \end{equation}
\end{theorem}
For the explicit expressions for $R$ and $D$ we refer to Section~\ref{sec:vertical-regularity}. In dimensions $n\notin\{2,4,8\}$ we obtain the following two-part result.

\begin{theorem}[{Proven in Section~\ref{sec:singularity-3D}}]
\label{highdsing}
    In dimensions $n\notin\{2,4,8\}$,
    \begin{enumerate}
        \item \label{item:separate-nbhood} for every real isotropic stiffness tensor $c_0$ there is a Euclidean open neighborhood $U(c_0)$ of real stiffness tensors such that if $c\in U(c_0)$, then the qP-branch of the slowness surface of $c$ is separate and
        \item \label{item:singular-scheme} the slowness surface of every complex stiffness tensor is singular as a complex scheme.
    \end{enumerate}
\end{theorem}

\begin{remark}
    We note that the part~\fullref{item:singular-scheme} of Theorem~\ref{highdsing} does not say anything about the singularity of real slowness surfaces, nor about whether the qP-branch is separate or not.
\end{remark}

\subsubsection{Applications to Inverse Problems}

As a direct corollary of Theorems~\ref{hregmain} and~\ref{2dseparate} and the result in~\cite{IM2023} we obtain two results for inverse problems.
The first one concerns X-ray tomography, which arises as linearization of travel times between boundary points; see e.g.~\cite{ilmavirta2019integral} for ray tomography and linearizations in the Riemannian realm.

\begin{corollary}[{Proven in Section~\ref{sec:xrt}}]
\label{geodesicxray}
    Let $M=\Bar{B}(0;1)\setminus\Bar{B}(0;R)\subset\R^n$ for any $R\in (0,1)$ and $n\ge 2$ and suppose $c$ is a spherically symmetric stiffness tensor field such that $F^c$ satisfies the Herglotz condition. If $c\in C^3(M)$ and the qP-branch of the slowness surface of $c$ is globally separate, then the geodesic X-ray transform of $(M,F^c_{qP})$ is injective on smooth scalar fields.
\end{corollary}

The second one concerns travel time data, which corresponds physically the collection of boundary arrival times from interior source points measured on the planet's surface.
See~\cite{dHILS2023} for more details and~\cite{ilmavirta2025lipschitz} for analogous results in the smooth category.

\begin{corollary}[{Proven in Section~\ref{sec:travel-time}}]
\label{traveltime}
    Let $M_1$ and $M_2$ be two smooth, connected and compact manifolds with boundary. Let $c_1$ and $c_2$ be stiffness tensor fields on $M_1$ and $M_2$ respectively such that $(M_i,F^{c_i}_{qP})$ are simple for $i\in\{1,2\}$. Then if the fields $c_i$ satisfy the conditions of Theorem~\ref{hregmain} for some $k\ge 2$ and the travel time data of $(M_i,F^{c_i}_{qP})$ coincide, there is a $C^{k+1}$-Finslerian isometry $\phi:M_1\to M_2$ fixing the boundary.
\end{corollary}

Both corollaries only use qP-polarized waves, which is the fastest-propagating one.
The local regularity results are also valid for qS, but the qS branch of the slowness surface may fail to be convex and thus fail to correspond to any Finsler geometry.
Also, it is common that the qP branch is globally separate from other ones but the other branches intersect each other.
Therefore qP is usually the only polarization that gives rise to nice global geometry.

\subsection{Algebraic Notions of Singularities and Implications for Analysis}

There are several ways to ask whether a slowness surface is smooth or singular.
We give the definitions and their comparisons below.
It is worth noting that different degrees of smoothness as a manifold need not be compared because as an algebraic variety the slowness surface is as a manifold either smooth or fails to be $C^1$.

An important distinction for us is that between varieties and schemes.
Schemes are more fine-grained algebraic structures and capture concepts like multiplicity better than varieties.
Comparison between algebra and geometry may be helpful:
varieties (in our context) are simply sets in an ambient Euclidean space and are in one-to-one correspondence with radical ideals of the polynomial ring over the ambient space, whereas schemes have more structure than being a mere subset and are in one-to-one correspondence with all ideals of the same ring.
Not all ideals are radical, so schemes are a generalization of varieties.
See Section~\ref{sec:singular} for details.

\begin{definition}\label{def:types-of-singularities}
Let $\K\in\{\R,\C\}$.
The vanishing set of a polynomial $f\colon\K^n\to\K$ is said to be
\begin{enumerate}
\item
\emph{smooth as a scheme} if there are no points where $f=0$ and $\nabla f=0$,
and
\item
\emph{smooth as a variety} if the vanishing set of the radical $\rad(f)$ is smooth as a scheme.
\end{enumerate}
\end{definition}

We will define the related concepts of the radical of a polynomial and the radical of an ideal in Section~\ref{sec:singular}.
At this time it suffices to know that $\rad(f)$ is the minimal polynomial for describing the vanishing set of $f$; for example $\rad(x^2y) = xy$.

The slowness surface is, of course, defined by the single polynomial $P_c$, the slowness polynomial.
We will explain in Section~\ref{sec:singular} why our definitions of smoothness for a variety or a scheme agree with the typical ones in our setting.
%We have tried to word the definitions in a maximally transparent way to an analytic audience.

%Each concept of smoothness comes with a related concept of singular points.
%For example, a singular point of the slowness surface as a scheme is a point where the polynomial and its gradient both vanish.

\emph{Degeneracy points} are vectors $p\in\K^n$ where the Christoffel matrix $\Gamma_c(p)$ has a multiple eigenvalue.
These are related to singular points of the slowness surface where smoothness (as defined above) fails.

\begin{proposition}[{Comparison of concepts of singularity for real slowness surfaces; proven in Section~\ref{sec:interpretation-sing-pt}}]
\label{realsingularpoints}
%\label{prop:sing-real}
Let $c$ be a real positive definite stiffness tensor and let $\Sigma=V(P_c)$ associated real slowness surface. Then the following hold.
\begin{enumerate}
\item
\label{item:real-scheme-to-variety}
If $\Sigma$ is smooth as a scheme, then it is smooth as a variety.
\item
\label{item:real-square-free}
Suppose $P_c$ is squarefree. Then $\Sigma$ is smooth as a scheme if and only if it is smooth as a variety.
% Smoothness of $\Sigma$ as a scheme and as a variety are equivalent if and only if $P_c$ is squarefree in the following sense:
% \begin{enumerate}
%     \item Suppose $P_c$ is squarefree. Then $\Sigma$ is smooth as a scheme if and only if it is smooth as a variety.
%     \item Suppose $P_c$ is not squarefree. Then $\Sigma$ is not smooth as scheme.
% \end{enumerate}
\item
\label{item:real-manifold-variety}
$\Sigma$ is smooth as a real manifold if and only if it is smooth as a variety.
\item
\label{item:real-scheme-degen-points}
Singular points of $\Sigma$ as a scheme are exactly the same as the degeneracy points of the Christoffel matrix.
%\item
%If a stiffness tensor field is smooth, then the elastic wave operator is of real principal type at all points where the slowness surface is smooth as a scheme.
\end{enumerate}
\end{proposition}

For point~\fullref{item:real-square-free} there is a partial converse:
There is a stiffness tensor $c$ (namely any isotropic one) in dimensions $n\geq3$ so that $P_c$ contains a square and the associated slowness surface is smooth as a variety but not as a scheme.

\begin{proposition}[{Comparison of concepts of singularity for complex slowness surfaces; proven in Section~\ref{sec:interpretation-sing-pt}}]\label{complexsingularpoints}
Let $c$ be a real positive definite stiffness tensor and let $\Sigma$ (the slowness surface) be the complex vanishing set of the slowness polynomial $P_c$.
\begin{enumerate}
\item
\label{item:complex-scheme-to-variety}
If $\Sigma$ is smooth as a scheme, then it is smooth as a variety.
\item
\label{item:complex-square-free}
If $P_c$ is squarefree, then smoothness of $\Sigma$ as a scheme and as a variety are equivalent.
If $P_c$ is not squarefree, then $\Sigma$ is not smooth as a scheme.
\item
\label{item:complex-manifold-variety}
$\Sigma$ is smooth as a complex manifold if and only if it is smooth as a variety.
\end{enumerate}
\end{proposition}

%\begin{remark}
The above results concern affine varieties and schemes.
The results would be different for complexified and projectivized versions of slowness surfaces.
For example, in 2D every isotropic slowness surface is singular because two circles of different radii intersect at the complex infinity.
%\end{remark}

\begin{remark}
Points on the cotangent bundle where the elastic wave operator is of \emph{real principal type} are points where the operator behaves microlocally well.
The crucial feature of real principal type is that the dimension of the kernel of the principal symbol of the operator is locally constant. A sufficient condition for this is that the slowness surface is smooth as a scheme. The implication will not in general work in the other direction, with the isotropic stiffness tensors serving as a counterexample.
See~\cite{Dencker1982} for a detailed definition.
While the other concepts are defined fiber by fiber, being real principal type is a local property on the cotangent bundle.
\end{remark}

An important example is the isotropic stiffness tensor.
The slowness polynomial is
\begin{equation}
\label{eq:isotropic-slowness-polynomial}
P(p)
=
(c_P^2\abs{p}^2-1)
(c_S^2\abs{p}^2-1)^{n-1}
,
\end{equation}
where $c_P$ and $c_S$ are the pressure and shear wave speeds.
The slowness surface is the union of two spheres (with radii $c_P^{-1}$ and $c_S^{-1}$), and this is smooth as a variety and from the analytic point of view.
But as the shear sphere is counted with multiplicity $n-1$, the isotropic slowness surface is not smooth as a scheme in dimensions $n\geq3$.
The slowness polynomial contains a square (a factor with an exponent ${}\geq2$).

The radical of the isotropic slowness polynomial in~\eqref{eq:isotropic-slowness-polynomial} is
\begin{equation}
\rad(P)(p)
=
(c_P^2\abs{p}^2-1)
(c_S^2\abs{p}^2-1)
.
\end{equation}
In fact, we conjecture that at least in dimension 3 a slowness polynomial corresponding to a positive definite stiffness tensor contains a square if and only if it is isotropic.
We know no other examples of this behavior.

\subsection{Motivation, Context and Related Results}

The ray paths of seismic waves can be modeled as integral curves of the Hamiltonian flows of the eigenvalues of the Christoffel matrix. Under the assumption that the qP-branch of the slowness surface is globally separate, the Hamiltonian flow of the qP-eigenvalue is the co-geodesic flow of the elastic wave geometry. For more details and background see for example~\cite{antonelli2003seismic,de2023determination,yajima2009finsler}. Another approach to modeling the wavefronts of elastic waves comes from microlocal analysis. For example for strictly hyperbolic systems the singularities of solutions propagate along the Hamiltonian flows of the eigenvalues of the principal symbol. Details can be found for example in~\cite{hintz2025introduction}. For a physical version of propagation of singularities along the ray paths see~\cite{Cerveny2001}. For example in two dimensions the isotropic elastic wave equation is a strictly hyperbolic system. 

The geometric point of view allows for the use of many powerful techniques of differential geometry in investigations of injectivity and stability for inverse problems related to the elastic wave equation. The existing theory of microlocal analysis fails already at a relatively high finite regularity, but we can still take the geometric point of view as a model for the physics at lower regularity. The applicability of the classical theory of Finsler geometry requires that the Finsler function is fiberwise smooth and at least of class $C^{1,1}$ along the manifold, so that the geodesic equation is uniquely solvable everywhere. This motivates the main results of our article, since in order to reliably apply the geometric model to a physical situation, we need to determine what we have to assume from the stiffness tensor in order for the geometric objects and results to remain valid. 

There is an extensive literature on inverse problems related to elasticity. In the isotropic case the resulting wave geometries are Riemannian, allowing one to apply the results on inverse problems on Riemannian manifolds. For example, based on Riemannian results it was shown in~\cite{rachele2000inverse} that the wave speeds are uniquely determined by the Dirichlet to Neumann map.  For more on the Riemannian results we refer to~\cite{paternain2023geometric} and references therein.

It is however the case that in applications the materials are very often anisotropic and nonsmooth. Hence there is a great need for generalizations to the Finslerian and nonsmooth settings. Regarding results on inverse problems in the nonsmooth Riemannian setting we mention~\cite{ilmavirta2021pestov,ilmavirta2024tensor}, where the authors prove injectivity results for metrics of class $C^{1,1}$ and~\cite{ilmavirta2023microlocal}, where the metric is assumed more regular, but the scalar field recovered by the geodesic X-ray transform is only required to be in $L^2$.

For results on ray transforms in the smooth Finslerian setting we mention~\cite{assylbekov2018x}, where the authors prove a variety of injectivity results for various curve families over a Finslerian surface and~\cite{dairbekov2008rigidity}, where injectivity results are obtained for ray transforms on closed Finsler manifolds of dimension greater than or equal to two. The injectivity of the Finslerian geodesic X-ray transform on was proved on simple manifolds of arbitrary dimension $\ge 2$ in~\cite{Ivanov2013}. There is also a number of works applying Finsler geometric methods to inverse problems in anisotropic elasticity, related for example to the Dix problem~\cite{de2025reconstruction} and to broken scattering relations~\cite{de2022foliated}. For previous work applying algebraic geometry to inverse problems in anisotropic elasticity we refer to~\cite{dHILVA2023}. 

There is a large body of work in inverse problems in elasticity not utilising Finslerian methods, related to for example unique continuation~\cite{de2023quantitative}, to elastic surface waves~\cite{de2019semiclassical}, to cloaking~\cite{quadrelli2021elastic}, to inverse source problems~\cite{bao2018inverse} and to scattering of elastic waves~\cite{bao2018direct}.

Elastic geometry is Finsler geometry, but very few Finsler geometries arise from elasticity.
Ivanov~\cite{Ivanov2013} gave a counterexample to boundary rigidity on Finsler manifolds, but the example does not seem to stay in the class of elastic Finsler metrics.
Therefore care is needed in modelling on the correct level of generality; Riemannian geometry is too narrow but fully general Finsler geometry is too broad.

\subsection{Organization of the Article}

The rest of the paper is organized as follows. Section~\ref{sec:preliminaries} is devoted to preliminaries, divided into those related to elasticity and to related to Finsler geometry. The required prerequisites related to algebraic geometry will be reviewed in Section~\ref{sec:vertical-regularity}. Section~\ref{sec:horizontal-regularity} is dedicated to the analytical results concerning the horizontal regularity and contains the proof of Theorem~\ref{hregmain}, while Section~\ref{sec:vertical-regularity} is dedicated to the algebraic results concerning the vertical regularity and contains the proofs of Theorems~\ref{2dseparate} and~\ref{highdsing}. The applications to inverse problems are treated in Section~\ref{sec:inverse-problems} and lastly Section~\ref{sec:singular} contains technical details on the definitions of singularities of schemes and varieties along with the proofs of Propositions~\ref{realsingularpoints} and~\ref{complexsingularpoints}.
Appendices~\ref{app:anisotropic-spaces} and~\ref{xraytechnicalappendix} contain proofs of the basic properties of our anisotropic function spaces and proofs of some of the more technical lemmas.

\subsection{Acknowledgments}
J.I. was supported by the Research Council of Finland
(Flagship of Advanced Mathematics for Sensing Imaging and Modelling grant 359208; Centre of Excellence of Inverse Modelling and Imaging grant 353092; and other grants 351665, 351656, 358047, 360434) and a Väisälä project grant by the Finnish Academy of Science and Letters.
P.K was supported by the Finnish Ministry of Education and Culture’s Pilot for Doctoral Programmes (Pilot project Mathematics of Sensing, Imaging and Modeling) and the Research Council of Finland
(Flagship of Advanced Mathematics for Sensing Imaging and Modelling grant 359208; Centre of Excellence of Inverse Modelling and Imaging grant 353092).
A.K. was supported by the Geo-Mathematical Imaging Group at Rice University.
We thank Maarten V. de Hoop for discussions.

\section{Geometric preliminaries}
\label{sec:preliminaries}

We divide the preliminaries into those dealing with Elasticity and those dealing with Finsler geometry.

\subsection{Elasticity}

An element $c=c_{ijkl}$ of $\R^{n^4}$ is a stiffness tensor, if it has the elastic symmetries $c_{ijkl}=c_{jikl}=c_{klij}$. We note that for us the stiffness tensors are not tensorial, but are interpreted as sections of the trivial bundle $M\times\R^{n^4}$, where $M\subset\R^n$ is the closure of a bounded smooth domain in $\R^n$ modeling the elastic body.
The dimension of the stiffness tensor is the dimension of the manifold on which it is defined, and the space of $n$-dimensional stiffness tensors with components taking values in $\K\in\{\R,\C\}$ is denoted by $E_{\K}(n)$ and can be identified with $\K^{N(n)}$, where $N(n)=\frac{n(n+1)}{2}$. The two dimensional subspace of isotropic $\K$-valued stiffness tensors will be denoted by $I_\K(n)$.  A stiffness tensor field $c:M\to E_{\K}(n)$ is defined to be of class $C^k(M)$, $k\in\mathbb{N}_{\ge 0}$, if the component functions are of class $C^k(M)$. 

We will denote the space of symmetric $n\times n$ matrices over $\K$ by $\text{Symm}_{\K}(n)$. The Christoffel matrix field associated to a stiffness tensor field $c:M\to E_{\K}(n)$ will be denoted by $\Gamma_c$, and is understood either as a map
\begin{equation}
    \Gamma_c: T^*M\to \text{Symm}_{\R}(n)
\end{equation}
or as a map
\begin{equation}
    \Gamma_c:T_\C^*M\to \text{Symm}_{\C}(n),
\end{equation}
depending on the situation. Here $T_\C^*M$ refers to the complexified cotangent bundle with the typical fibers $T_p^*\otimes_{\R}\C$. In local coordinates $(x,p)$ of $T^*M$ or $T^*_\C M$ the Christoffel matrix has the representation
\begin{equation}
    (\Gamma_c)_{il}(x,p)=c_{ijkl}(x)p_jp_k.
\end{equation}
The characteristic polynomial of the Christoffel matrix $\Gamma_c(p)$ is called the slowness polynomial and will be denoted by $P_c(p)$.
The leading eigenvalue of the Christoffel matrix, corresponding to the qP-waves, will be denoted by $\lambda_{qP}^c$ and the smaller eigenvalues, corresponding to the qS-waves, will be denoted by $\lambda_{qS,j}^c$ for $j\in\{1,\dots,n-1\}$. We are not ruling out the possibility that $\lambda_{qP}^c=\lambda_{qS,1}^c$. The Finsler functions associated to $\lambda_{qP}^c$ and $\lambda_{qS,j}^c$ will be denoted by $F_{qP}^c$ and $F_{qS,j}^c$ respectively. In two dimensions we will omit the index $j$.

\subsection{Finsler Geometry}

We recall the basic definitions of Finsler geometry. For more details we refer to the books~\cite{antonelli2003handbook,BCS2000,shen2001lectures}, and for a quick introduction to~\cite{dahl2006brief}. Let $M$ be a smooth manifold. A function $F:TM\to [0,\infty)$ is a Finsler function if for each $x\in M$ and $X,Y,Z\in T_xM$ we have
\begin{itemize}
    \item positive definiteness: $F(x,X)=0\iff X=0$, 
    \item positive 1-homogeneity along fibers: $F(x,\lambda X)=\lambda F(x,X)$ for all $\lambda >0$ and
    \item strong convexity: $F(x,\dummy):T_xM\to [0,\infty)$ is smooth outside the zero section and for all $X\neq 0$ the symmetric bilinear form
    \begin{equation}
        g_{x,X}(Y,Z):=\frac{1}{2}\frac{\partial^2}{\partial t\partial s}(F(x,X+sY+tZ)^2)\vert_{s=t=0}
    \end{equation}
    is positive definite. The form $g$ is referred to as the fundamental tensor. 
\end{itemize}
A Finsler manifold is a pair $(M,F)$, where $F$ is a Finsler function. A Finsler function will always be of class $C^1$ at the zero section, but due to homogeneity $F^2$ is smooth at the zero section if and only if 
\begin{equation}
    F(x,y)=\sqrt{g_{ij}(x)y^iy^j},
\end{equation}
where $g$ is a Riemannian metric.

    \begin{remark}
        Some authors replace the strong convexity condition with the subadditivity assumption
        \begin{equation}
            F(x,X+Y)\leq F(x,X)+F(x,Y)
        \end{equation}
        for all $X,Y\in T_xM$. Since strong convexity implies subadditivity and the invertibility --- and hence positive-definiteness --- of the matrix of the fundamental tensor is required for writing the geodesic equation in the standard form, we choose to take strong convexity as a part of the definition.
    \end{remark}
    
\subsubsection{Metric Structure}

Let $\infty<a<b<\infty$ and let $\gamma:[a,b]\to M$ be a smooth curve. The length $\ell(\gamma)$ of $\gamma$ is defined as
    \begin{equation}
        \ell(\gamma)=\int_a^bF(\gamma(t),\Dot{\gamma}(t))\ dt.
    \end{equation}
Due to the homogeneity of the Finsler function the length of a curve is invariant under positively oriented reparamerizations. Assuming $M$ is path connected, the Finsler function lets us define a distance on $M$ by declaring for any $x,y\in M$ the distance
    \begin{equation}
        d(x,y)=\inf\{\ell(\gamma):\gamma:[0,1]\to M\ \text{smooth and}\ \gamma(0)=x,\ \gamma(1)=y\}.
    \end{equation}
We note that unless we require 
    \begin{equation}\label{reversibility}
        F(x,-\dummy)=F(x,\dummy)
    \end{equation} 
on each $T_xM$, we do not necessarily have $d(x,y)=d(y,x)$. A Finsler function $F$ satisfying \eqref{reversibility} fiberwise is called reversible, and the distance function of a reversible Finsler function defines a metric on the manifold.

\subsubsection{Geodesics}

The Euler-Lagrange equation associated to the length functional $\ell$ can be written, assuming strong convexity, in components as
\begin{equation}\label{geodesiceq}
    \Ddot{\gamma}^i+\frac{1}{2}g^{ij}(\gamma,\Dot{\gamma})\left(2\frac{\partial g_{jk}}{\partial x^l}(\gamma,\Dot{\gamma})-\frac{\partial g_{kl}}{\partial x^j}(\gamma,\Dot{\gamma})\right)\Dot{\gamma}^k\Dot{\gamma}^l=0,
\end{equation}
where $g_{ij}$ are the components of the fundamental tensor and $[g^{ij}(x,y)]_{i,j=1}^n$ is the inverse matrix of $[g_{ij}(x,y)]_{i,j=1}^n$.
The solutions to \eqref{geodesiceq} are called geodesics. The lifts $(\gamma,\Dot{\gamma})$ of geodesics are integral curves of the so-called geodesic spray $\mathbb{G}\in \Gamma(T(TM\mO))$, given in local coordinates by
\begin{equation}
    \mathbb{G}(x,y)=y^i\frac{\partial}{\partial x^i}+2G^i(x,y)\frac{\partial}{\partial y^i},
\end{equation}
where the spray coefficients $G^i$ are given by
\begin{equation}
    G^i(x,y)=\frac{1}{4}g^{ij}(x,y)\left(2\frac{\partial g_{jk}}{\partial x^l}(x,y)-\frac{\partial g_{kl}}{\partial x^j}(x,y)\right)y^ky^l
    .
\end{equation}

\section{Horizontal Regularity}
\label{sec:horizontal-regularity}

A characteristic feature of the elastic wave geometry is that the Finsler function enjoys much greater regularity on the fibers than along the manifold. In order to capture this phenomena, we introduce a class of anisotropic function spaces. To motivate the specific definition of our anisotropic function spaces, we take an interlude to talk about Legendre transforms and how the elastic Finsler function $F^c$ is constructed.

\subsection{Legendre Transform and Construction of the Elastic Finsler Function}

Any smooth function $H:T^*M\mO\to \R$ with a fiberwise invertible Hessian, given in local coordinates $(x,p)$ of $T^*M\mO$ by
\begin{equation}
    \text{Hess}(H)^{ij}(x,p)=\frac{\partial^2H}{\partial p_i\partial p_j}(x,p),
\end{equation}
defines an invertible map $\ell_H:T^*M\mO\to TM\mO$ via the local coordinate formula
\begin{equation}
    \ell_H(x,p)=(x,\partial_{p_1}H(x,p),\dots,\partial_{p_n}H(x,p)).
\end{equation}
The Legendre transformation $\mathcal{L}_H(f)$ of any $f\in C^\infty(T^*M\mO)$ is then defined to be the smooth function $\mathcal{L}_H(f):TM\mO\to\R$ defined by
\begin{equation}
    \mathcal{L}_H(f)=(\ell_H^{-1})^*E_f,
\end{equation}
where the energy $E_f$ of the function $f$ is defined in the local induced coordinates $(x,p)$ of $T^*M\mO$ by
\begin{equation}
    E_f(x,p)=p_i\frac{\partial f}{\partial p_i}(x,p)-f(x,p).
\end{equation}

Now if the Hamiltonian is fiberwise convex and 2-homogeneous, then $\mathcal{L}_H(H)$ is also fiberwise convex and 2-homogeneous. Hence given a Finsler function $\Tilde{F}$ on $T^*M$, we can then obtain a Finsler function $F$ on $TM$ by the formula
\begin{equation}
    \frac{1}{2}F^2=\mathcal{L}_{\frac{1}{2}\Tilde{F}^2}\left(E_{\frac{1}{2}\Tilde{F}^2}\right)
\end{equation}
by defining 
\begin{equation}
    \ell_{\frac{1}{2}\Tilde{F}^2}(0)=0.
\end{equation}
Due to the fiberwise homogeneity of $\Tilde{F}$, we have by Euler's theorem (see e.g.~\cite[Theorem 1.2.1]{BCS2000}) that
\begin{equation}
    E_{\frac{1}{2}\Tilde{F}^2}=\frac{1}{2}\Tilde{F}^2
\end{equation}
and hence the equation for the Finsler function $F$ simplifies to
\begin{equation}
    \frac{1}{2}F^2=(\ell_{\frac{1}{2}\Tilde{F}^2}^{-1})^*\frac{1}{2}\Tilde{F}^2,
\end{equation}
which then gives us
\begin{equation}
    F=(\ell_{\frac{1}{2}\Tilde{F}^2}^{-1})^*\Tilde{F}.
\end{equation}

The qP-Finsler function of a stiffness tensor $c$ is obtained as follows. First we use the implicit function theorem to write the qP-eigenvalue $\lambda_{qP}^c$ of the Christoffel matrix $\Gamma_c$ as a function on the cotangent bundle. The function $(\lambda_{qP}^c)^{1/2}$ can be seen to define a Finsler function on $T^*M$. For details we refer to~\cite{dHILS2023}. This is then used to construct the elastic Finsler function $F^c_{qP}$ on $TM$ via the Legendre transformation induced by the function $\frac{1}{2}\lambda_{qP}^c$. More specifically, by the previous discussion we have the formula
\begin{equation}
    F^c_{qP}=(\ell_{\frac{1}{2}\lambda_{qP}^c}^{-1})^*(\lambda_{qP}^c)^{1/2}.
\end{equation}

Hence to determine the anisotropic regularity of the qP-geometry arising from a stiffness tensor $c$, we need to know how anisotropic regularity is inherited by implicit and inverse functions. We have designed our anisotropic regularity classes so that they are preserved exactly by implicit and inverse functions. 

\subsection{Anisotropic Regularity}

Let $V\subset\R^m$ and $U\subset\R^n$ be open. Denote by $\partial_v$ the derivatives with respect to the variables in $V$ and by $\partial_u$ the derivatives with respect to the variables in $U$. Define 
\begin{equation}
    C^k(V\times U;\R^d)
\end{equation}
to be the set of functions $f:V\times U\to \R^d$ that are $k$-times continuously differentiable with respect to the first variable and by
\begin{equation}
    C_l(V\times U;\R^d)
\end{equation}
the set of functions $f:V\times U\to\R^d$ that are $l$-times continuously differentiable with respect to the second variable. 

\begin{definition}[Anisotropic regularity on Euclidean domains]
The anisotropic regularity class $C^k_l(V\times U;\R^d)$, $1\leq k\leq l$, is defined to be the set of functions $f:V\times U\to\R^d$ for which
\begin{itemize}
    \item $f\in C^{k}(V\times U;\R^d)\cap C_l(V\times U;\R^d)$,
    \item $\partial_u^\alpha f\in C^{\min\{k,l-\vert\alpha\vert\}}(V\times U;\R^d)$ for all $\alpha$ with $\vert\alpha\vert\leq l$,
    \item $\partial_v^\alpha f\in C_{l-\vert\alpha\vert}(V\times U;\R^d)$ for all $\alpha$ with $\vert \alpha\vert\leq k$.
\end{itemize}
    
\end{definition}

The following elementary Lemmas will be of great use. The proof of Lemma~\ref{commutingderivatives} can be found in Appendix~\ref{app:anisotropic-spaces} and we omit the proof of Lemma~\ref{Anisotropiccomposition} as an elementary consequence of the chain rule.

\begin{lemma}\label{commutingderivatives}
    If $f\in C^k_l(V\times U;\R^d)$, we have that
    \begin{equation}
        \partial_v^\alpha\partial_u^\beta f=\partial_u^\beta\partial_v^\alpha f
    \end{equation}
    for all $\alpha$ and $\beta$ with $\vert\alpha\vert\leq k$ and $\vert\alpha\vert+\vert\beta\vert\leq l$.
\end{lemma}

\begin{lemma} \label{Anisotropiccomposition}
    Let $U,V\subset\R^n$ and $l\ge 2$. Suppose $f\in C^k_{l}(V\times U; \R^n)$ and $g\in C^k_l(V\times f(V\times U);\R)$. Then
    \begin{equation}
        g\circ f\in C^k_{l}(V\times U;\R).
    \end{equation}
\end{lemma}

Let then $X\times Y$ be a smooth product manifold. We define the regularity class $C^k_l(X\times Y;M)$ as follows.

\begin{definition}[Anisotropic regularity on product manifolds]
    We say that $f\in C^k_l(X\times Y;M)$, if for each $p\in X\times Y$ there are local charts $(V_p,x)$ and $(U_p,y)$ of $X$ and $Y$ respectively such that $\proj_X(p)\in V_p$ and $\proj_Y(p)\in U_p$ and a local chart $(W_{f(p)},z)$ of $M$ such that $f((V_p\times U_p)\cap (X\times Y))\subset W_{g(p)}$ we have that
        \item $$z\circ f\circ (x^{-1}\times y^{-1})\in C^k_l(x(V_p)\times y(U_p);z(W_{f(p)})).$$
\end{definition}

\begin{remark}
    By Lemma~\ref{Anisotropiccomposition} anisotropic regularity in a single local chart at a point implies anisotropic regularity in all local charts at a point.
\end{remark}

Let then $E\to B$ be a smooth fiber bundle with a typical fiber $F$. The regularity class $C^k_l(E;M)$ is defined via local trivializations:

\begin{definition}[Anisotropic regularity on fiber bundles]
    Let $\{(U_i,\varphi_i)\}_{i\in I}$ be a bundle atlas of $E\to F$. We say that $f\in C^k_l(E;M)$ if for any local trivialization $\varphi_i$ of $E$ we have that $$f\circ\varphi_i^{-1}\in C^k_l(U_i\times F;M).$$
\end{definition} 

\begin{remark}
    For completeness, in the case of smooth $G$-bundles by Lemma~\ref{Anisotropiccomposition} we have that anisotropic regularity is preserved under the action of the structure group and is hence well defined. In the cases relevant to our applications this just means that anisotropic regularity is preserved by the local transition functions of $TM$ and $T^*M$.
\end{remark}

\subsubsection{Outline of the Proof of Theorem~\ref{hregmain}}

The proof of Theorem~\ref{hregmain} will proceed in three steps. First we prove that anisotropic regularity is preserved by inverse functions and then that fiberwise smooth anisotropic regularity is preserved by implicit functions. We then use the result for inverse functions to determine the anisotropic regularity of the Legendre transformation associated to a given function. The result then follows by combining the above.

\subsection{Anisotropic Regularity of Inverse Functions}

We begin with the anisotropic regularity of inverse functions.

\begin{proposition}\label{inverseregonbundles}
    Let $M$ be a smooth manifold and $(A,M,\pi_A,F_A)$ and $(B,M,\pi_B,F_B)$ be two smooth fiber bundles. Suppose $f:A\to B$ is a $C^1$ bundle diffeomorphism and moreover $f\in C^k_l(A;B)$ for $1\leq k\leq l$. Then $f^{-1}\in C^k_l(B;A)$.
\end{proposition}
The result is a relatively direct application of the following result, the proof of which can be found in Appendix~\ref{app:anisotropic-spaces}.

\begin{lemma} \label{Anisotropicinverse}
    Suppose $V\subset\R^m$ and $U\subset\R^n$ are open and $f:V\times U\to\R^n$ is such that 
    \begin{itemize}
        \item $f\in C^k_\ell(V\times U;\R^n)$,
        \item $f(v,U)=f(w,U)=:f(U)$ for all $v,w\in V$,
        \item for each $v\in V$ the function $f(v,\dummy)$ is a $C^\ell$-diffeomorphism between $U$ and $f(U)$.
    \end{itemize}
    Then
    \begin{equation}
        f^{-1}\in C^k_\ell(V\times f(U);\R^n).
    \end{equation}
\end{lemma}

\begin{proof}[Proof of Proposition~\ref{inverseregonbundles}]
    Pick $p\in B$ and a neighborhood $\Omega_B\subset B$ of $p$. By assumption there is some neighborhood $\Omega_A\subset A$ of $f^{-1}(p)$ such that $\Omega_B=f(\Omega_A)$. Since $f$ covers the identity, we have that
    \begin{equation}
        \pi_B(\Omega_B)=\pi_A(\Omega_A)=:U.
    \end{equation}
    Pick a local trivialization $\varphi_B$ of $B$ at $p$ and a local trivialization $\varphi_A$ of $A$ at $f^{-1}(p)$. After possibly choosing $\Omega_A$ smaller, we may assume that $(U,x)$ is a local chart on $M$. Take a local chart $(V_B,y_B)$ of $F_B$ at $\proj_{F_B}(\varphi_B(p))$ and a local chart $(V_A,y_A)$ of $F_A$ at $\proj_{F_A}(\varphi_A(f^{-1}(p)))$. Then the map
    \begin{equation}\label{map}
        (x\times y_B)\circ\varphi_B\circ f\circ \varphi_A^{-1}\circ (x^{-1}\times y_A^{-1}):x(U)\times y_A(V_A)\to x(U)\times y_B(V_B)
    \end{equation}
    can be written in coordinates as
    \begin{equation}
        (x,y_A)\mapsto (x,\Tilde{f}(x,y_A)),
    \end{equation}
    where the map $\Tilde{f}:x(U)\times y_A(V_A)\to y_B(V_B)$ is determined by \eqref{map}. Since $f\in C^k_l(A;B)$, we have by definition that $\Tilde{f}\in C^k_l(x(U)\times y_A(V_A);y_B(V_B))$. Since the map $f$ is a bundle diffeomorphism, we have that $\Tilde{f}(x,y_A(V_A))=y_B(V_B)$ for all $x\in x(U)$. Since $f$ is a $C^1$-diffeomorphism, so is $\Tilde{f}$ and since $\Tilde{f}\in C_l(x(U)\times y_A(V_A))$, by the inverse function theorem $\Tilde{f}(x,\dummy)$ is a $C^l$-diffeomorphism onto $y_B(V_B)$ for all $x\in x(U)$. Hence the conditions of Lemma~\ref{Anisotropicinverse} are satisfied and we get that $\Tilde{f}^{-1}\in C^k_l(x(U)\times y_B(V_B);y_A(V_A))$.

    Since the map
    \begin{equation}
        (x,y_B)\mapsto (x,\Tilde{f}(\Tilde{f}^{-1}(x,y_B))
    \end{equation}
    gives the identity map of $x(U)\times y_B(V_B)$ and the map
    \begin{equation}
        (x,y_A)\mapsto (x,\Tilde{f}^{-1}(\Tilde{f}(x,y_A)))
    \end{equation}
    gives the identity map of $x(U)\times y_A(V_A)$, we can conclude that
    \begin{equation}
        \begin{split}
            &\left((x\times y_B)\circ\varphi_B\circ f\circ \varphi_A^{-1}\circ (x^{-1}\times y_A^{-1})\right)^{-1}(x,y_B) \\
            &=\left((x\times y_A)\circ \varphi_A\circ f^{-1}\circ \varphi_B^{-1}\circ (x^{-1}\times y_B^{-1})\right)(x,y_B) \\
            &=(x,\Tilde{f}^{-1}(x,y_B))
        \end{split}
    \end{equation}
    for all $(x,y_B)\in x(U)\times y_B(V_B)$. Hence
    \begin{equation}
        (x\times y_A)\circ \varphi_A\circ f^{-1}\circ \varphi_B^{-1}\circ (x^{-1}\times y_B^{-1})\in C^k_l(x(U)\times y_B(V_B);x(U)\times y_A(V_A)),
    \end{equation}
     letting us conclude that
    %\begin{equation}
    $
        f^{-1}\in C^k_l(B;A)
    $.
    %\end{equation}
\end{proof}

\subsection{Anisotropic Regularity of Implicit Functions}

We move on to the anisotropic regularity of implicit functions. We start with the following result, the proof of which can be found in Appendix~\ref{app:anisotropic-spaces}.

\begin{lemma} \label{anisotropicimplicit}
    Let $V\times U\times W\subset \R^n\times\R^m\times \R^d$ be an open set and let $F:V\times U\times W \to \R^d$ be in $C^k(V\times U\times W;\R^d)$ with $k\ge 1$ such that for each $p\in V\times U$ we have $F(p,\dummy)\in C^\infty(W;\R^d)$.
    Denote variables in $V,U$ and $W$ by $v,u$ and $w$ respectively.

    Suppose there is a continuously differentiable function $\phi:V\times U\to W$ for which we have for all $v,u\in V\times U$ that
    \begin{equation}
        F(v,u,\phi(v,u))=0\quad \text{and}\quad  \det(J_wF(v,u,w))\vert_{(v,u,\phi(v,u))}\neq 0.
    \end{equation}
    Then, if
    \begin{itemize}
        \item the functions $\partial^\alpha_uF$ are $k$-times continuously differentiable in $v$ for all $\alpha$,
        \item the function $J_wF$ is $k$ times continuously differentiable with respect to $v$ and $C^\infty$-smooth with respect to $u$ and
        \item the functions $\partial_u^{\alpha}J_wF$ are $k$ times continuously differentiable with respect to $v$ for all $\alpha$,
    \end{itemize}
    we have that $\phi\in C^k_\infty(V\times U;W)$.
\end{lemma}

For implicit functions we prove the following less general result.

\begin{proposition}\label{anisotropicimplicitmain}
    Let $(E,M,\pi_E,F_E)$ be a smooth fiber bundle, $W\subset \R^d$ an open set and let $F\in C^k(E\times W;\R^d)$ be a function such that $F(p,\dummy)\in C^\infty(W;\R^d)$. Suppose that there is a function $\phi\in C^1(E;W)$ such that
    \begin{equation}
        F(p,\phi(p))=0\quad \text{and}\quad \det(J_wF(p,w))\vert_{(p,\phi(p))}\neq 0
    \end{equation}
    for all $p\in E$. If for each $p\in E$ there is local bundle chart $(U,\varphi)$ at $p$ with $x$ the coordinate function on $U$ and a local coordinate chart $(V,y)$ at $\proj_{F_E}(\varphi(p))$ such that the function
    \begin{equation}
        (x,y,z)\mapsto F((\varphi^{-1}\circ (x^{-1}\times y^{-1}))(x,y),z)
    \end{equation}
    satisfies the conditions of Lemma~\ref{anisotropicimplicit} in $x(U)\times y(V)\times W$, then $\phi\in C^k_\infty(E;W)$.
\end{proposition}

\begin{proof}
    Let $p\in E$. By the assumptions there is a bundle chart $(U,\varphi)$ at $p$ with $x$ a coordinate function on $U$, and coordinate chart $(V,y)$ of $F_E$ at $\proj_{F_E}(\varphi(p))$ such that
    \begin{equation}
        \phi\circ \varphi^{-1}\circ (x^{-1}\times y^{-1})\in C^k_\infty(x(U)\times y(V);\R^d).
    \end{equation}
    The claim follows by definition.
\end{proof}

\subsection{Anisotropic Regularity Under Legendre Transforms}

Before discussing the anisotropic regularity of Legendre transforms, we give the following lemma.
\begin{lemma} \label{helplemma 1}
    Suppose $H:T^*M\mO\to\R$ is fiberwise strictly convex and $H\in C^k_{\ell}(T^*M\mO)$ for any $1\leq k\leq l$ and $\ell\ge 2$. Then
    \begin{equation}
        \ell_H:T^*M\mO\to TM\mO
    \end{equation}
    is a $C^1$ bundle diffeomorphism and $\ell_H\in C^k_{l-1}(T^*M\mO;TM\mO)$.
\end{lemma}
\begin{proof}
     That $\ell_H$ is a $C^1$ diffeomorphism between $T^*M_0$ and $TM$ under the prevailing regularity assumptions is classical and we refer to~\cite{dahl2006brief} for a proof. It hence suffices to establish anisotropic regularity.
     
     Let $p\in T^*M\mO$ and let $(U,\varphi)$ be a bundle chart of $T^*M\mO$ at $p$, induced by the smooth structure of $M$. Denote by $\pi'$ the bundle map of $T^*M\mO$ and by $(x,\xi)$ the local coordinates of $T^*M\mO$ in $\varphi((\pi')^{-1}(U))$. Let $(U,\psi)$ be a bundle chart of $TM\mO$ at $\ell_H(p)$, induced by the smooth structure of $M$ and denote by $\pi$ the bundle map of $TM\mO$ and by $(x,y)$ the local coordinates of $TM\mO$ at $\psi(\ell_H(p))$. Then
    \begin{equation}
        \Tilde{\ell}_H:=(x\times y)\circ \psi\circ \ell_H\circ \varphi^{-1}\circ (x^{-1}\times \xi^{-1})=(x,\partial_{\xi_1}H(x,\xi),\dots,\partial_{\xi_n}H(x,\xi))
    \end{equation}
    in $\varphi((\pi')^{-1}(U))=x(U)\times(\R^n\setminus\{0\})$. Since by assumption
    \begin{equation}
        H\circ \varphi^{-1}\circ(x^{-1}\times\xi^{-1})\in C^k_l(x(U)\times(\R\setminus\{0\}),
    \end{equation} 
    we get that $\Tilde{\ell}_H\in C^k_{l-1}(x(U)\times (\R^n\setminus\{0\});x(U)\times y(\psi(\ell_H(p))))$. The claim follows.
\end{proof}

With this we obtain the following result.
\begin{proposition} \label{Anisotropiclegendre}
    Let $\Tilde{F}:T^*M\mO\to\R$ be a Finsler function of regularity class $C^k_l(T^*M\mO)$ with $1\leq k\leq l$ and $l\ge 3$. Define $\ell_{\frac{1}{2}\Tilde{F}^2}(0)=0$. Then
    \begin{equation}
        F:=(\ell_{\frac{1}{2}\Tilde{F}^2}^{-1})^*\Tilde{F}
    \end{equation}
    is a Finsler function on $TM$ and is of class $C^k_{l-1}(TM\mO)$.
\end{proposition}
\begin{proof}
    The convexity and homogeneity of $F$ is classical and for the proof we refer to~\cite{dahl2006brief}. Hence it suffices to establish regularity.
    
    By Lemma~\ref{helplemma 1}, if $\Tilde{F}\in C^k_l(T^*M\mO)$, then $\ell_{\frac{1}{2}\Tilde{F}^2}\in C^k_{l-1}(T^*M\mO;TM\mO)$. Since $\ell_{\frac{1}{2}\Tilde{F}^2}$ is a $C^1$-bundle diffeomorphism, by Proposition~\ref{inverseregonbundles} we have that $(\ell_{\frac{1}{2}\Tilde{F}^2})^{-1}\in C^k_{l-1}(TM\mO;T^*M\mO)$. The claim follows by Lemma~\ref{Anisotropiccomposition}.
\end{proof}

It is worth noting that the Legendre transformation may reduce vertical regularity, since the definition involves taking first order vertical derivatives.

\subsection{Proof of the Main Result on Horizontal Regularity}
\label{sec:thm-1-proof}

We have now developed sufficient tools in order to prove Theorem~\ref{hregmain}.

\begin{proof}[Proof of Theorem~\ref{hregmain}]
    The Finsler function $(\lambda_{qP}^c)^{1/2}$ satisfies by definition in the local coordinates $(x,p)$ of $T^*M\mO$ the equation
    \begin{equation}
        \chi(\Gamma_c)(x,p)=\text{det}(\Gamma_c(x,p)-\lambda_{qP}^c(x,p)\text{Id})=0,
    \end{equation}
    where $\Gamma$ is the Christoffel matrix. The function $\chi(\Gamma_c)$ is polynomial in $\lambda$ and is hence smooth in $\lambda$. By assumption the root $\lambda_{qP}^c(x,p)$ is simple and hence for each $(x,p)\in T^*M\mO$ we have that
    \begin{equation}
        \partial_\lambda\chi(\Gamma_c(x,p,\lambda))\vert_{(x,p,\lambda_{qP}^c(x,p))}\neq 0.
    \end{equation}
    Since $c\in C^k(M)$ and $\chi(\Gamma_c)$ is polynomial also in the fiber coordinates, the conditions of Proposition~\ref{anisotropicimplicitmain} are satisfied. Hence we have that $\lambda_{qP}^c\in C^k_\infty(T^*M\mO)$, which then by Proposition~\ref{Anisotropiclegendre} implies that \begin{equation}
        F^c_{qP}=(\ell_{\frac{1}{2}\lambda_{qP}^c}^{-1})^*(\lambda_{qP}^c)^{1/2}\in C^k_\infty(TM\mO).
    \end{equation}
    This concludes the proof.
\end{proof}

\section{Vertical Regularity}
\label{sec:vertical-regularity}

We now move on to discuss vertical regularity. Let $\K\in\{\R,\C\}$ and $p\in\C^n$ and define the set $\Sigma_\K(n)$ of singular stiffness tensors as
\begin{equation}
    \Sigma_{\K}(n)=\{c\in E_{\K}(n):P_c\ \text{has a multiple root for some}\ p\in \K^n\setminus\{0\}\}.
\end{equation}

\begin{remark}
    By Proposition~\ref{realsingularpoints}, in the real case the singularity of a slowness surface means precisely the existence of a multiple eigenvalue of the Christoffel matrix.
\end{remark}

The aim of this chapter is to prove Theorems~\ref{2dseparate} and~\ref{highdsing}. We start with the two dimensional result.

\subsection{Singularity of Slowness Surfaces in Two Dimensions}
\label{sec:singularity-2D}

The main result of this subsection is Proposition~\ref{prop:2d-slowness-surf-sing}, which characterizes singular stiffness tensors in $\R^2$. Theorem~\ref{2dseparate} then follows as an immediate corollary.  The characterization is in terms of certain semi-algebraic conditions on the stiffness tensor defined by the polynomials
\begin{equation}
\label{eqn:discs-and-resus}
\begin{split}
    D_1(c)
    &=
    (c_{1212} + c_{1122})^2
    -
    4c_{1112}c_{1222},
    \\
    D_2(c)
    &=
    (c_{1112} - c_{1222})^2
    +
    (c_{1111} - c_{1212})
    (c_{2222} - c_{1212}),
    \\
    D(c) &= D_1(c) + D_2(c),
    \\
    L(c)
    &=
    c_{1122}c_{1222}
    +
    c_{1111}c_{1222}
    -
    c_{1112}c_{2222}
    -
    c_{1122}c_{1112},
    \\
    R(c) &= L(c)^2 - D_1(c)D_2(c).
\end{split}
\end{equation}

\begin{proposition}
\label{prop:2d-slowness-surf-sing}
The following are equivalent for a stiffness tensor $c \in E_\R(2)$:
\begin{enumerate}
    \item \label{item:multiple-eigen} The Christoffel matrix has a multiple eigenvalue.
    \item \label{item:d-r-posit} $R(c) = 0$ and $D(c) \geq 0$.
    \item \label{item:d1-r-posit} $R(c) = 0$ and $D_1(c) \geq 0$.
    \item \label{item:d2-r-posit} $R(c) = 0$ and $D_2(c) \geq 0$.
\end{enumerate}
In particular,
\begin{equation}
\Sigma_{\R}(2)
=
\{
c \in E_\R(2)
:
R(c) = 0
\text{ and }
D(c) \geq 0
\}.
\end{equation}
%In addition,
%\begin{equation}
%\Sigma_{\C}(2)
%=
%\{
%c \in E_\C(2)
%:
%R(c) = 0
%\}
%\end{equation}
%and
%\begin{equation}
%\Sigma_{\R}(2)
%=
%\{
%c \in A(2)
%:
%R(c) = 0
%\text{ and }
%D(c) \geq 0
%\}.
%\end{equation}
\end{proposition}

\begin{proof}
The existence of a multiple root of a polynomial $Q$ is equivalent to the vanishing of the discriminant $\Delta(Q)$ of $Q$.
The discriminant of the characteristic polynomial of $H: = \Gamma_c(p)$ is
\begin{equation}
\Delta(\chi_H)
=
(h_{11} - h_{22})^2
+
4h_{12}^2.
\end{equation}
Thus $\Delta(\chi_H) = 0$ can be equivalently written as
\begin{equation}
h_{11} - h_{22} = 0
\quad\text{and}\quad
h_{12} = 0.
\end{equation}
Spelled out we have the system of equations
\begin{equation}
\label{eqn:disc-char-poly2}
c_{1112}p_1^2
+
(c_{1212} + c_{1122})p_1p_2
+
c_{1222}p_2^2
=
0
\end{equation}
and
\begin{equation}
\label{eqn:disc-char-poly1}
(c_{1111} - c_{1212})p_1^2
+
2(c_{1112} - c_{1222})p_1p_2
+
(c_{1212} - c_{2222})p_2^2
=
0.
\end{equation}
Let us denote by $F_1$ and $F_2$ the polynomials in $\R[p_1,p_2]$ defining the equations \eqref{eqn:disc-char-poly1} and \eqref{eqn:disc-char-poly2} respectively.

We will prove that equations \eqref{eqn:disc-char-poly1} and \eqref{eqn:disc-char-poly2} have a common non-zero root $p \in \R^n$ if and only if any one of the conditions~\fullref{item:d-r-posit},~\fullref{item:d1-r-posit} or~\fullref{item:d2-r-posit} are true. This is done by analyzing the discriminants and the resultant of the system of equations. More concretely, we use the resultant of $F_1$ and $F_2$ to study existence of a common root and the discriminants of $F_1$ and $F_2$ to determine whether a common root is real.

Since $F_1$ and $F_2$ are homogeneous our conclusions are the same independently whether we choose to interpret $F_1,F_2 \in (\R[p_1])[p_2]$ and require $p_1 \ne 0$ or $F_1,F_2 \in (\R[p_2])[p_1]$ and require $p_2 \ne 0$ when computing the discriminants and the resultant. We choose to use the first interpretation.

We interpret $F_1(p_1,p_2)$ as a polynomial in the variable $p_2$ with coefficients in $\R[p_1]$. Then the discriminant of $F_1$ is
\begin{equation}
\label{eqn:disc1}
\Delta(F_1)(p_1)
=
[
(c_{1212} + c_{1122})^2
-
4c_{1112}c_{1222}
]
p_1^2
=
D_1(c)p_1^2.
\end{equation}
Similarly, we interpret $F_2$ as a polynomial in the variable $p_2$ with coefficients in $\R[p_1]$ and compute its discriminant to be
\begin{equation}
\label{eqn:disc2}
\Delta(F_2)(p_1)
=
[
4(c_{1112} - c_{1222})^2
+
4
(c_{1111} - c_{1212})(c_{2222} - c_{1212})
]
p_1^2
=
4D_2(c)p_1^2.
\end{equation}
Additionally, with $F_1$ and $F_2$ interpreted in this way their resultant is the polynomial
\begin{equation}
\label{eqn:resultant}
\text{res}(F_1,F_2,p_2)(p_1)
=
(L(c)^2 - D_1(c)D_2(c))p_1^4
=
R(c)p_1^4.
\end{equation}
where $L(c)$ is as in \eqref{eqn:discs-and-resus}.

The proof is now completed by the following reasoning:
\begin{enumerate}
    \item From \eqref{eqn:resultant} we see that the slowness polynomial of $\Gamma_c(p)$ has a double root at $p \in \C^2 \setminus \{0\}$ if and only if $R(c) = 0$.
    \item From \eqref{eqn:disc1} and \eqref{eqn:disc2} we see that $p \in \R^2 \setminus \{0\}$ if and only if $D_1(c) \ge 0$ or $D_2(c) \ge 0$.
    \item From the vanishing of the resultant we see that $D_1(c)$ and $D_2(c)$ have the same sign.
    \item In the case that $D_1(c) \ge 0$ and $D_2(c) \ge 0$ we have $D(c) \ge 0$. In the case that $D_1(c) < 0$ and $D_2(c) < 0$ we have $D(c) < 0$.
    \qedhere
\end{enumerate}
\end{proof}

\begin{proof}[Proof of Theorem~\ref{2dseparate}]
The claim is the same as equivalence of items~\fullref{item:multiple-eigen} and~\fullref{item:d-r-posit} in Proposition~\ref{prop:2d-slowness-surf-sing}.
\end{proof}

\subsection{Singularity of Slowness Surfaces in Higher Dimensions}
\label{sec:singularity-3D}

In this subsection we prove Theorem~\ref{highdsing}. Before going into the details, we have to take a slight detour to discuss a modified notion of parallelizability, which we will refer to as projective parallelizability. 

\subsubsection{Projective Parallelizability}

An $n$-dimensional smooth manifold $M$ is said to be parallelizable if there are $n$ vector fields $V^1,\dots,V^n$ on it so that they form a basis of $T_xM$ at all points $x\in M$.
Such a collection of vector fields corresponds to a trivialization of the tangent bundle.
We will study a weaker form of parallelizability:

\begin{definition}
\label{def:proj-par}
An $n$-dimensional smooth manifold $M$ is said to be \emph{projectively parallelizable} if there are $n$ tangent line bundles $L^1,\dots,L^n$ so that at all points $x\in M$ we have $T_xM=\bigoplus_{i=1}^nL^i_x$.
\end{definition}

A tangent line bundle is a one-dimensional subbundle of the tangent bundle.
If a manifold $M$ is parallelizable, then it is easily seen to be projectively parallelizable by setting $L^i=V^i\R$.
The converse is only partially true.
If each tangent line bundle $L^i$ admits a nowhere vanishing section, then the manifold is parallelizable with said sections as the required vector fields.

This definition arises naturally in our study of non-degenerate matrices.
The tangent line bundles correspond to eigenspaces of certain operators, and eigenspaces being more unique objects than eigenvectors turns our attention to our projective kind of parallelizability.
We prove two results concerning this definition:

\begin{proposition}
\label{prop:proj-par-cover}
Every projectively parallelizable manifold has a parallelizable covering space.
If the original manifold is connected or compact, so is the covering space, respectively.
\end{proposition}

\begin{proposition}
\label{prop:proj-par-sphere}
The sphere $\mathbb{S}^n$ is projectively parallelizable if and only if $n\in\{1,3,7\}$.
\end{proposition}

Proposition~\ref{prop:proj-par-cover} has an immediate generalization we will not need:
It is unimportant that the number of tangent line bundles or vector fields equals the dimension.
The statement remains true for any number $k\in\{1,\dots,n\}$ of these bundles and fields.
For $k=1$ this leads to an ``unoriented hairy ball theorem'', which we record here without proof:

\begin{proposition}
The sphere $\mathbb{S}^n$ has a tangent line bundle if and only if $n$ is odd.
\end{proposition}

\begin{proof}[Proof of Proposition~\ref{prop:proj-par-cover}]
Let $M$ be a connected projectively parallelizable space with the tangent line bundles $L^1,\dots,L^n$.
Endow $M$ with any Riemannian metric.

Define the space
\begin{equation}
\begin{split}
N
:=
\{
(x,v^1,\dots,v^n);
x\in M,\ 
v^i(x)\in L^i_x \text{ for all }i,\text{ and}\ 
\vert v^i(x)\vert = 1 \text{ for all }i
\}
.
\end{split}
\end{equation}
This $N$ is a $2^n$-sheeted covering space of $M$ with the canonical projection.

The space $N$ is locally diffeomorphic to $M$ and it has globally defined vector fields $v^i(x)$ obtained by pulling back the vector field on $M$ along the covering mapping. As each tangent space of $M$ is by assumption the direct sum of the tangent lines $L^1,\dots,L^n$, these vector fields provide a basis of each tangent space of $N$.
Therefore $N$ is parallelizable.

This concludes the proof of the main claim.
If $M$ is compact, so is $N$.
If $M$ is connected, any connected component of $N$ will do.
\end{proof}

The proof is analogous to that of an unoriented manifold having an orientable double cover.
The unit vectors on the lines are unique up to sign, and switching any of the signs is a deck transformation of the cover.
With any choice of signs the vector fields are locally well defined and unique, and the point of the space $N$ is to make these fields globally well defined.

\begin{proof}[Proof of Proposition~\ref{prop:proj-par-sphere}]
It is well known that the sphere $\mathbb{S}^n$ is parallelizable if and only if $n\in\{1,3,7\}$.
Thus it remains to prove that for $n\geq2$ the sphere $\sphere^n$ is parallelizable if (and only if) it is projectively parallelizable.
This follows quickly from Proposition~\ref{prop:proj-par-cover}: a projectively parallelizable sphere has a parallelizable cover, but the spheres (for $n\geq2$) are their own universal covers and cannot be covered by anything else, thus proving that the sphere itself is parallelizable.
\end{proof}

\subsubsection{Proof of Theorem~\ref{highdsing}}

Coming back to the main aim of the subsection, we start by proving the Zariski-closedness of $\Sigma_\C(n)$.

\begin{lemma}\label{zclosed}
    The set $\Sigma_{\C}(n)$ is Zariski-closed.
\end{lemma}

Due to the homogeneity of $\Delta(P_c)$ in $p$ it suffices for us to investigate the set
\begin{equation}
    B:=\{(c,p)\in E_{\C}(n)\times \C\mathbb{P}^{n-1}:\Delta(P_c(p))=0\}.
\end{equation}

\begin{remark}
    The homogeneity of the discriminant in $p$ can be seen as follows. Since the matrix $\Gamma(c,p)$ is homogeneous in $p$, so are its eigenvalues. Since the discriminant of the characteristic polynomial $\chi_{\Gamma(c,p)}=P_c(p)$ is given by
    \begin{equation}
        \prod_{i<j}(\lambda_i-\lambda_j)^2
    \end{equation}
    where $\lambda_k$, $k\in\{1,\dots,n\}$, are the eigenvalues of $\Gamma(c,p)$,  we see that $\Delta(P_c(p))$ is homogeneous in $p$ as well.
\end{remark}

Let $\pi:B\to E_{\C}(n)$ be the projection
\begin{equation}
    \pi:(c,p)\mapsto c.
\end{equation}
We note the following, which follows directly from the definitions.

\begin{lemma}\label{piBisSigmaP}
    It holds that $\pi(B)=\Sigma_{\C}(n)$.
\end{lemma}

With this we are then ready to give the 

\begin{proof}[Proof of Lemma~\ref{zclosed}]
    Since the set $B$ is closed, by Lemma~\ref{piBisSigmaP} it follows by the fundamental theorem of elimination theory~\cite{CLO2015} that the set $\Sigma_\C(n)$ is closed as well.
\end{proof}

We then move on to finding the Euclidean interior point in the real slice $\Sigma_\C(n)\cap\R^N$.

\begin{proposition}\label{interiorpoint}
    For $n\notin\{2,4,8\}$ the real slice $\Sigma_{\C}(n)\cap\R^{N}$ has a Euclidean interior point.
\end{proposition}

The proof relies on the following result.

\begin{lemma}\label{Almostradial}
\label{lma:almost-isotropic}
Let $c_0\in E_{\C}(n)\cap\R^N$ be any isotropic stiffness tensor and let $\varepsilon>0$ be any positive number.
There is a Euclidean neighborhood $\Omega\subset E_{\C}(n)\cap\R^N$ of $c_0$ so that the following hold:
\begin{enumerate}
\item
\label{item:non-degen-eigen}
For all $c\in\Omega$ and all $p\in\R^n\setminus\{0\}$ the eigenspace $E_1(c,p)$ of the highest eigenvalue of $\Gamma_c(p)$ is non-degenerate.
\item
\label{item:dist-to-eigen}
For all unit vectors $p\in\R^n$ we have $d(p,E_1(c,p))<\varepsilon$.
\end{enumerate}
\end{lemma}

\begin{proof}
    Choose an isotropic real stiffness tensor $c_0\in E_{\C}(n)\cap\R^N$ and let $p\in\mathbb{S}^{n-1}\subset\R^{n}$ be a real unit vector. Denote by $N(p)$ the outer unit normal field on $\mathbb{S}^{n-1}$. For each $p\in\mathbb{S}^{n-1}$, the outer normal vector $N(p)$ of $\mathbb{S}^{n-1}$, identified with $p$, is an eigenvector of the Christoffel matrix $\Gamma(c_0,p)$ and corresponds to a simple largest eigenvalue $\lambda_1(p)$ of $\Gamma(c_0,p)$.
    
    Since the leading eigenvalue $\lambda_1$ is simple, by the implicit function theorem there is for each $\Tilde{p}\in\mathbb{S}^{n-1}$ a neighborhood $U_{\Tilde{p}}$ of $(c_0,\Tilde{p})$ in $\R^N\times\R^n$ such that in $U_{\Tilde{p}}$ the leading eigenvalue of $\Gamma(c,p)$ remains simple, is continuous and the corresponding eigenvector $Z(c,p)$ can be chosen to depend continuously on $(c,p)$. Hence for any $\varepsilon>0$ we may --- by the compactness of $\mathbb{S}^{n-1}$ --- choose a finite number $M(\varepsilon)\in\mathbb{N}_{\ge 1}$ of points $p_i\in\mathbb{S}^{n-1}$ and some $r>0$ such that
    \begin{itemize}
        \item $\mathbb{S}^{n-1}\subset\bigcup_{i=1}^{M(\varepsilon)}B(p_i;\varepsilon/2)$ and
        \item if $c\in B(c_0;r\varepsilon)$ then for each $i\in\{1,\dots,M(\varepsilon)\}$ we have $\vert Z(c,p)-p_i\vert<\frac{1}{2}\varepsilon$ whenever $p\in B(p_i;\varepsilon/2)$.
    \end{itemize}
    This gives us~\fullref{item:non-degen-eigen}, since we can cover $\mathbb{S}^{n-1}$ with neighborhoods where the leading eigenvalue is simple and continuous.
    
    To deduce~\fullref{item:dist-to-eigen}, we note that by the above if $c\in B(c_0;r\varepsilon)$, then for each $p\in\mathbb{S}^{n-1}$ the Christoffel matrix $\Gamma(c,p)$ has an eigenvector $Z(c,p)\in E_1(c,p)$ such that for some $i\in\{1,\dots,M(\varepsilon)\}$ we get
    \begin{equation}
        \vert Z(c,p)-p\vert\leq \vert p-p_i\vert+\vert Z(c,p)-p_i\vert <\varepsilon.
    \end{equation} 
    Hence $d(p,E_1(c,p))<\varepsilon$ for all $p\in\mathbb{S}^{n-1}$.
\end{proof}

We can now prove Proposition~\ref{interiorpoint}.

\begin{proof}[Proof of Proposition~\ref{interiorpoint}]
   Pick any isotropic $c_0\in E_{\R}(n)$.
Let $\varepsilon=\frac{1}{2}$ and let $\Omega\subset E_{\C}(n)$ be the set given by Lemma~\ref{Almostradial}.
Take any $c\in\Omega$.
We will show that $c\in\Sigma_{\C}(n)$.

Suppose that $c\notin \Sigma_{\C}(n)$.
This means that for all unit vectors $p\in \mathbb{S}^{n-1}\subset\R^n$ the Christoffel matrix $\Gamma(c,p)$ has $n$ positive and distinct eigenvalues.
Let $E_k(c,p)$ denote the eigenspace of the $k$th highest eigenvalue.

The important property of these eigenspaces is that $p\notin E_k(c,p)$ for $k\ge 2$.
To see this, fix $k\ge 2$ and write $p=\sum_ip_i$ so that $p_i\in E_i(c,p)$.
The properties provided by Lemma~\ref{Almostradial} ensure that
\begin{equation}
\sum_{i=2}^{n}p_i^2
=
d(p,E_1(c,p))^2
<\varepsilon^2
=
\frac{1}{4},
\end{equation}
and so $\vert p_k\vert<\frac{1}{2}$.
Any unit vector $v\in E_k(c,p)$ satisfies
\begin{equation}
\vert p\cdot v\vert
=
\vert p_k\vert
<
\frac{1}{2},
\end{equation}
so $p$ cannot be parallel to $v$.
Thus indeed $p\notin E_k(c,p)$.

This allows us to project the radial direction out smoothly.
That is, the projection $\R^n\ni v\mapsto v-(v\cdot p)p\in\R^n$ is injective when restricted to $E_k(c,p)$.
Therefore the space
\begin{equation}
\hat E_k(c,p)
:=
\{
v\mapsto v-(v\cdot p)p
;
v\in E_k(c,p)
\}
\end{equation}
is one-dimensional and orthogonal to $p$.
This happens to $p$, and the space $\hat E_k(c,p)$ depends smoothly on $p$ --- which means that one can locally write a spanning vector of this space as a smooth function of $p$.

This means that the spaces $\hat E_j(c,p)\subset T_pS^{n-1}$, $j\in\{2,\dots,n\}$ are smoothly varying one-dimensional subspaces and thus they make up smooth tangent line bundles of $S^{n-1}$. Since
\begin{equation}
    T_p\mathbb{S}^{n-1}=\bigoplus_{j=2}^n\hat{E}_j(c,p)
\end{equation}
for all $p\in\mathbb{S}^{n-1}$, we get that $\mathbb{S}^{n-1}$ is projectively parallelizable. This is however in contradiction with Proposition~\ref{prop:proj-par-sphere} for $n\in\{2,4,8\}$.
Therefore $c\in \Sigma_{\R}(n)$ as desired.
\end{proof}

Theorem~\ref{highdsing} now follows by combining the above with the following result.

\begin{lemma}\label{realintpointzariskiclosed}
    Let $m\ge 1$ and suppose the real slice $S\cap \R^m$ of a Zariski-closed set $S\subset\C^m$ has a Euclidean interior point. Then $S=\C^m$.
\end{lemma}
\begin{proof}
    Suppose $P$ is a polynomial vanishing on $S$. Since $S\cap\R^m$ has a Euclidean interior point, the polynomial $P\vert_{\R^m}$ vanishes on an open set. Hence $P\vert_{\R^{m}}=0$ by analyticity.

    Suppose then that $P\vert_{\R^{m}}=0$. Write
    \begin{equation}
        P(z_1,\dots,z_{m})=\sum_{j=0}^{N}z_1^jQ_j(z_2,\dots,z_N),
    \end{equation}
    where $Q_j$ are polynomials and $N$ denotes the edegree of $P$ in $z_1$. We get for any $(x_2,\dots,x_m)\in\R^{m-1}$ that 
    \begin{equation}
        P(z_1,x_2,\dots,x_m)
    \end{equation}
    is a polynomial in $z_1$ vanishing on the real line and hence the zero polynomial in the first argument. This shows that for each $j$ the polynomial $Q_j$ is a polynomial on $\C^{m-1}$ vanishing on $\R^{m-1}$ and hence that the polynomial $P$ vanishes on $\C\times \R^{m-1}$. Repeating the same argument for each $Q_j$ shows that $Q_j=0$ in $\C\times \R^{m-2}$ and hence that the polynomial $P$ vanishes on $\C^2\times\R^{m-2}$. Iterating this process tells us that the polynomial $P$ vanishes in all of $\C^m$, giving us the claim.
\end{proof}

We are now in the position to prove Theorem~\ref{highdsing}.

\begin{proof}[Proof of Theorem~\ref{highdsing}] Part~\fullref{item:separate-nbhood} follows from Lemma~\ref{Almostradial}.

To see part~\fullref{item:singular-scheme}, we note that by Proposition~\ref{interiorpoint} the set $\Sigma_{\C}(n)$ is a Zariski closed set with a Euclidean interior point in $\Sigma_{\C}(n)\cap \R^N$. Hence by Lemma~\ref{realintpointzariskiclosed}, it holds that $\Sigma_{\C}(n)=\C^N$.
\end{proof}

\section{Applications to Inverse Problems}
\label{sec:inverse-problems}

Moving on to the applications, we start with Theorem~\ref{geodesicxray} and then move on to discuss the travel time problem.

\subsection{Injectivity of the Geodesic X-Ray Transform}
\label{sec:xrt}

The smooth version of Corollary~\ref{geodesicxray} was given in~\cite{IM2023}.
The proof below is a small variation of the proof given therein, with focus on regularity matters to be checked.

\begin{proof}[Proof of Corollary~\ref{geodesicxray}]
If our stiffness tensor $c$ is of class $C^3(M)$ and the qP-branch of $c$ is globally separate, by Theorem~\ref{hregmain} the associated elastic wave geometry is of class $C^3_\infty(TM_0)$. In this regularity the definitions of the geodesic X-Ray transform and the simplicity of a Finsler manifold make sense without change and the Herglotz condition can be stated as in~\cite{IM2023}. The result follows by verifying that the proofs of the key Lemmas 3.1 and 3.2  in~\cite{IM2023} remain valid in this regularity. This is shown in Appendix~\ref{xraytechnicalappendix}.
\end{proof}

\subsection{Travel Time Problem}
\label{sec:travel-time}

We begin by proving a generalization of the low regularity Myers--Steenrod-type result presented in~\cite{matveev2017myers} to the case of a manifold with boundary.

\begin{proposition}[Low Regularity Myers--Steenrod Theorem for Compact Manifolds with Boundary] \label{Low regularity Myers-Steenrod}

Let $(M_1,F_1)$ and $(M_2,F_2)$ be two compact, connected and reversible Finsler manifolds with a smooth boundary, where the Finsler metrics $F_1$ and $F_2$ are of class $C^{k}$ with $k\ge 2$. Suppose $\phi:M_1\to M_2$ is a distance preserving bijection. Then the map $\phi$ is a diffeomorphism of class $C^{k+1}$ and moreover $F_1=\phi^*F_2$.
    
\end{proposition}

\begin{remark}
    We need the assumption $k\ge 2$ since our proof relies on the use of boundary normal coordinates, which makes sense only in the case that the geodesic equation is uniquely solvable; a fact we can quarantee to hold only in the case $k\ge 2$.
\end{remark}

We need the following lemma, a straightforward generalization of~\cite[Lemma 3.3]{de2023determination}.

\begin{lemma} \label{boundaryexpregularity}
    Let $(M,F)$ be as defined above. There exists $\varepsilon>0$ such that $\text{exp}^\perp:[0,\varepsilon)\times \partial M\to M$ is a $C^{k-1}$-diffeomorphism onto its image.
\end{lemma}
With this we move on to give the

\begin{proof}[Proof of Lemma~\ref{boundaryexpregularity}.]
    The map $\exp^\perp$ can be written as
    \begin{equation}
        \exp^{\perp}(s,z)=\pi(\phi_s(\nu_{\text{in}}(z)),
    \end{equation}
    where $\pi:TM\to M$ is the bundle map and $\phi$ is the geodesic flow of a $C^k$-extension $(\Tilde{M},\Tilde{F})$ of $(M,F)$. By implicit function theorem the field $\nu_{\text{in}}$ is of class $C^k$ and the geodesic flow is of class $C^{k-1}$ and hence $\text{exp}^\perp$ is of class $C^{k-1}$. The rest of the claim follows by the inverse function theorem and compactness as in~\cite{dHILS2023}.
\end{proof}

\begin{proof} [Proof of Proposition~\ref{Low regularity Myers-Steenrod}] To see that $\phi$ and $\phi^{-1}$ are of class $C^{k+1}$, we begin by noting that by the invariance of domain $\phi(\text{int}(M_1))=\text{int}(M_2)$. By~\cite{matveev2017myers} the claim holds for $\phi\vert_{\text{int}(M_1)}$. It therefore suffices to establish regularity near the boundary.

For $i\in\{1,2\}$ the Finslerian distance $d_i$ induces an intrinsic distance on each connected component of $\partial M_i$ that agrees with the Finslerian distance $\rho_i$ of the closed manifold $(\partial M_i,F_i\vert_{\partial M_i})$. Hence the restriction $\phi\vert_{\partial M_1}:\partial M_1\to \partial M_2$ is a metric isometry and thus by~\cite{matveev2017myers} a diffeomorphism of class $C^{k+1}$.

By Lemma~\ref{boundaryexpregularity} we have that for some $\varepsilon>0$ the boundary normal exponential map $\exp^{\perp}$ is a $C^{k-1}([0,\varepsilon)\times\partial M_1)$-diffeomorphism onto its image. This lets us define for $i\in \{1,2\}$ the boundary normal coordinates $\varphi$ as $\varphi_i(p):=(\text{exp}^\perp)^{-1}(p)=(s_i(p),z_i(p))$, where $s_i(p)=d(p,z_i(p))$ and $z_i(p)$ is the boundary point closest to $p$.

By the regularity of the restriction map $\phi\vert_{\partial M_1}$ the map
\begin{equation}
    [0,\varepsilon)\times \partial M_1\ni(s,z)\mapsto (s,\phi(z))\in [0,\varepsilon)\times \partial M_2
\end{equation}
is of class $C^{k+1}$. For each $p$ in the domain of $\varphi_1$ the fact that $\phi$ is a metric isometry implies that
\begin{equation}
    s_1(p)=s_2(\phi(p))\quad \text{and}\quad \phi(z_1(p))=z_2(\phi(p)).
\end{equation}
Hence
\begin{equation}
    \varphi_2\circ\phi\circ\varphi_1^{-1}(s,z)=(s,\phi(z)),
\end{equation}
showing that $\phi$ is of class $C^{k+1}$ near the boundary. The regularity claim follows for $\phi^{-1}$ by reversing the roles for $M_1$ and $M_2$.

The pullback property holds for $\phi\vert_{\partial M_1}$ by~\cite{matveev2017myers} and since $\phi$ is a diffeomorphism of the whole manifold with boundary, the continuity of the metrics imply the pullback property over the boundary as well.
\end{proof}

\begin{proof}[Proof of Corollary~\ref{traveltime}]
Follows directly from Proposition~\ref{Low regularity Myers-Steenrod}.
\end{proof}

\section{Detailed Analysis of Singular Points}
\label{sec:singular}

This section contains more detailed descriptions of varieties and schemes, and their smoothness.
We also prove Propositions~\ref{realsingularpoints} and~\ref{complexsingularpoints}.

\subsection{Ideals, polynomials, and radicals}

All our algebra will take place within the polynomial ring $R=\K[x_1,\dots,x_n]$ in $n$ real or complex variables ($\K=\R$ or $\K=\C$).
An ideal of this ring is a set $I\subset R$ such that $0 \in I$ and for which
$p+q\in I$ whenever $p\in I$ and $q\in I$
and also
$pq\in I$ whenever $p\in I$ and $q\in R$.

The ideal $\generate{S}$ generated by any subset $S\subset R$ is the smallest ideal containing $S$ and consists of all finite $R$-linear combinations of elements in $S$.
Our ring $R$ is Noetherian, so every ideal is generated by a finite set.

The radical $\sqrt{I}$ of an ideal $I$ consists of all $p\in R$ for which $p^n\in I$ for some $n\geq1$.
This $\sqrt{I}$ is also an ideal, and the ideal $I$ is said to be radical if $I=\sqrt{I}$.
The radical of an ideal is a radical ideal.

The polynomial ring $R$ is a unique factorization domain, so any single polynomial $p\in R$ can be factored (essentially) uniquely as a product $p=q_1^{m_1}\cdots q_k^{m_k}$ of different irreducible polynomials $q_i$ raised to some powers $m_i$.
The radical $\rad(p)$ of $p$ is $q_1\cdots q_k$.
The radical is only defined up to multiplication by non-zero constant, but this makes no difference for us.
For an ideal generated by a single polynomial we have $\sqrt{\generate{f}}=\generate{\rad(f)}$.

We say that a polynomial $f\in R$ is squarefree if it is not divisible by $p^2$ for any non-constant $p\in R$.
The ideal generated by $f$ is radical if and only if $f$ is squarefree.
Therefore the squarefreeness of the slowness polynomial plays a role in the regularity of the slowness surface.

\subsection{Smoothness of varieties and schemes}

The Jacobian criterion can be used as the definition of smoothness of a scheme over a field $\K \in \{\R,\C\}$. Next, we discuss how the general definition reduces to the one given in Definition~\ref{def:types-of-singularities}.

\begin{definition}[{\cite[Definition 13.2.4]{Vakil2025}}]
\label{def:scheme-singularity}
Let $\K \in \{\R,\C\}$. A $\K$-scheme is $\K$-smooth of dimension $d$, if it is of pure dimension $d$, and there exists a cover by affine open sets 
\begin{equation}
\Spec(\K[x_1,\dots,x_n]/\langle f_1,\dots,f_m\rangle)
\end{equation}
where the Jacobian matrix of the vector field $F = (f_1,\dots,f_m)$ has corank $d$ at all points.
\end{definition}

Pure dimension $d$ means that the irreducible components of the scheme are all of equal dimension $d$. All slowness surfaces in $\R^n$ are of pure dimension $n-1$.

Let $P \in \K[x_1,\dots,x_n]$ be a slowness polynomial. The corresponding slowness surface viewed scheme theoretically is the affine scheme
\begin{equation}
X = \Spec(\K[x_1,\dots,x_n]/\langle P \rangle).
\end{equation}
An obvious affine open cover of $X$ is formed by $X$ itself. It turns out that checking smoothness according to Definition~\ref{def:scheme-singularity} can be performed for any affine open cover and the conclusion will be the same for all other open covers (see~\cite[Section 13.1.7]{Vakil2025}). We choose to always use the trivial cover.

The underlying set of the scheme $X$ is the prime spectrum of the quotient ring i.e. the collection of all of its prime ideals.
Since the scheme theoretic definition of a slowness surface only involves a single polynomial the Jacobian matrix in the definition  reduces to the gradient of the slowness polynomial $P$. The corank of the gradient is $n-1$ at $p \in X$ if and only if the gradient does not vanish at the point $p$ on the slowness surface. Hence smoothness of the slowness surface as a scheme reduces exactly to the two conditions stated in Definition~\ref{def:types-of-singularities}: the vanishing of $P$ controls that we are looking at points on the surface and the additional vanishing of $\nabla P$ controls the regularity.

Consider any (non-contant) polynomial $f \in \K[x_1,\dots,x_n]$. Then the the polynomials $f$ and $f^2$ determine the same algebraic variety and hence the concept of smoothness as a variety is the same for both sets. However, the scheme determined by $f^2$, formally
\begin{equation}
\Spec(\K[x_1,\dots,x_n]/\langle f^2 \rangle),
\end{equation}
is always singular at all points, which simple to see: when $f^2 = 0$ then also $2f\nabla f = 0$. If $f$ is square free, then $\rad(f^2) = f$ and the corresponding scheme is smooth. Hence $\{f^2 = 0\}$ is smooth as a variety according to Definition~\ref{def:types-of-singularities}.

We end this section by considering how the general definition for smoothness of an algebraic variety reduces to the one we have given in Definition~\ref{def:types-of-singularities}.

\begin{definition}[{\cite[Chapter 6]{CLO2015}}]
Let $p$ be a point on an affine variety $V$. Then $p$ is said to be smooth provided $\dim T_pV = \dim_p V$. An affine variety is smooth if all of its points are smooth.
\end{definition}

There are two notions in the definition which need to be explained. The first is the dimension of a variety $V$ at a point $p$, denoted by $\dim_pV$, which simply means the highest dimension of an irreducible component of $V$ containing $p$.

The second notion to be explained is the tangent space $T_pV$. The tangent space of an affine variety $V$ at $p \in V$ is defined in pure algebraic terms by the vanishing of the linear part of the defining polynomials of the variety (see~\cite[Chapter 6]{CLO2015}). If $f \in \K[x_1,\dots,x_n]$ is a polynomial then its linear part at $p$, denoted by $d_pf$, is by definition
\begin{equation}
d_pf = \sum_{i=1}^n \frac{\partial f}{\partial x_i}(p)(x_i - p_i).
\end{equation}
The tangent space at $p$ is then the affine variety
\begin{equation}
T_pV
=
V(d_pf \,:\, f \in I(V))
\end{equation}
where $I(V)$ is the ideal in $\K[x_1,\dots,x_n]$ consisting of all polynomials vanishing on the set $V$. Since a polynomial ring over a field ($\K = \R$) is always a Noetherian ring, the ideal $I(V)$ is finitely generated i.e. there are finitely many polynomials $g_i$ so that $I(V) = \langle g_1,\dots,g_m \rangle$ and hence $T_pV = V(d_pg_1,\dots,d_pg_m)$.

It is a fact that the ideal $I(V)$ is always radical: if $f^n \in I(V)$ for some $n \in \N$ then $f \in I(V)$. In the context of slowness surfaces this has the following interpretation. A slowness surface as a variety is the vanishing set $V = V(P)$ of the slowness polynomial $P$. Then $I(V) = \sqrt{\langle P\rangle} = \langle\rad(P)\rangle$. A point $p$ on the slowness surface can be singular if and only if
\begin{equation}
\dim T_pV > \dim_pV
\end{equation}
which can only happen if the gradient of the polynomial $\rad(P)$ vanishes at $p$ (see~\cite[Chapter 6]{CLO2015}). In other words, the slowness surface (as a variety) can be singular if and only if the vanishing set of the radical $\rad(P)$ is not smooth as a scheme, exactly as we have defined in Definition~\ref{def:types-of-singularities}.

\subsection{Interpretation of Singular Points}
\label{sec:interpretation-sing-pt}

We give here the proofs of Propositions~\ref{realsingularpoints} and~\ref{complexsingularpoints}.

\begin{proof}[Proof of Proposition~\ref{realsingularpoints}]
     \fullref{item:real-scheme-to-variety} Suppose $\Sigma$ is smooth as a scheme. Then $\nabla P_c\neq 0$ and hence the set $\Sigma=V(P_c)$ is smooth as a manifold. By the claim \fullref{item:real-manifold-variety} this implies smoothness as a variety.
     
     \fullref{item:real-square-free} By squarefreeness $\rad(P_c) = P_c$. Thus the definitions are equivalent.
    
    \fullref{item:real-manifold-variety} Suppose $\Sigma$ is smooth as a variety. Then, since $\Sigma$ is the level set of a smooth function with a nonvanishing gradient, it is smooth as a manifold. 
    
    Suppose then that $\Sigma$ is smooth as a manifold. Since all slowness surfaces in $\R^n$ are of pure dimension $n-1$, we have that $\mathrm{dim}(\Sigma)=n-1$. Hence for each $p\in\Sigma$ we have $\mathrm{dim}(T_p\Sigma)=n-1$. Therefore, identifying $T_p\Sigma$ with the tangent plane to $\Sigma$ at $p$ in $\R^n$, we get
    \begin{equation}
        \mathrm{dim}(V(d_pP_c))=\mathrm{dim}(T_p\Sigma)=\mathrm{dim}_p\Sigma,
    \end{equation} 
    telling us that $\Sigma$ is smooth as a variety.
    
    \fullref{item:real-scheme-degen-points} Suppose $\K=\R$. Let $p$ be a multiple root of $P_c$. Let $v\in\mathbb{S}^{n-1}$ and consider the polynomial $P_{c,v}(s):=P_c(p+sv)$. Since $s=0$ is a multiple root of $P_{c,v}$ we get that
    \begin{equation}
        \frac{d}{ds}P_{c,v}(s)\vert_{s=0}=v\cdot \nabla P_c(p)=0.
    \end{equation}
    Since $v$ was arbitrary, this means that $\nabla P_c(p)=0$.

    Suppose then that $\nabla P_c(p)=0$. Then for any $v\in\mathbb{S}^{n-1}$ we have
    \begin{equation}
        \frac{d}{ds}P_{c,v}(s)\vert_{s=0}=v\cdot\nabla P_c(p)=0
    \end{equation}
    and since moreover $P_{c,v}(0)=0$, we have that $s=0$ is a multiple root of $P_{c,v}$, which then in turn implies that $p$ is a multiple root of $P_c$. 

    Since a real slowness polynomial $P_c$ has a multiple root at $p$ if and only if the Christoffel matrix has a degenerate eigenvalue at $p$, the claim follows. 
\end{proof}
The proof of the complex version is essentially the same.

\begin{proof}[Proof of Proposition~\ref{complexsingularpoints}]
%The proof for the claims (1) and (2) is the same as in the real case.
The proof for claim~\fullref{item:complex-scheme-to-variety} is the same as in the real case.

For claim~\fullref{item:complex-square-free} we amend the proof as follows.
If $P_c$ is squarefree, the proof is as in the real case above.
If $P_c$ is not squarefree, it is divisible by $q^2$ for some non-constant polynomial $q$.
The polynomial $q$ has a complex zero, and it will be a non-smooth one for $P_c$.

To prove claim~\fullref{item:complex-manifold-variety}, suppose first that $\Sigma$ is smooth as a variety. Then there is a polynomial $Q$ such that $\Sigma=Q^{-1}(0)\subset\C^n$ and $dQ\vert_{\Sigma}\neq 0$. This gives us the first implication, since preimages of regular values of holomorphic functions between complex manifolds are complex submanifolds. For details we refer to~\cite{lee2024introduction}. The second implication follows from the same argument as in the real case by replacing the field to $\C$ and dimension to complex dimension.
\end{proof}

\appendix

\section{Properties of Anisotropic Function Spaces}
\label{app:anisotropic-spaces}

Here we give the proofs of Lemmas~\ref{commutingderivatives},~\ref{Anisotropicinverse} and~\ref{anisotropicimplicit}.

\subsection{Proof of Lemma~\ref{commutingderivatives}}
Suppose $\vert\alpha\vert=\vert\beta\vert=1$. Let $\varphi\in C^\infty_0(V\times U;\R)$. Integrating by parts and the smoothness of $\varphi$ gives us
    \begin{equation}
        \begin{split}
            \int_{V\times U}\partial_{v_i}\partial_{u_j}f\varphi\ dx&=\int_{V\times U}f\partial_{u_j}\partial_{v_i}\varphi\ dx \\
            &=\int_{V\times U}f \partial_{v_i}\partial_{u_j}\varphi\ dx \\
            &=\int_{V\times U}\partial_{u_j}\partial_{v_i}f\varphi\ dx,
        \end{split}
    \end{equation}
    implying that
    \begin{equation}
        \int_{V\times U}(\partial_{v_i}\partial_{u_j}f-\partial_{u_j}\partial_{v_i}f)\varphi\ dx=0
    \end{equation}
    for all $\varphi\in C^\infty_0(V\times U;\R)$. Since we assume both $\partial_{v_i}\partial_{u_j}f$ and $\partial_{u_j}\partial_{v_i}f$ to exist continuously, by the fundamental lemma of the calculus of variations we obtain that 
    \begin{equation}
        \partial_{v_i}\partial_{u_j}f=\partial_{u_j}\partial_{v_i}f
    \end{equation}
    in $V\times U$ for all indices $i$ and $j$.

    Suppose then that $\partial_v^\alpha\partial_{u_j}f=\partial_{u_j}\partial_v^\alpha f$ for all $j\in \{1,\dots,m\}$ and for all $\alpha$ with $\vert\alpha\vert\leq r$, where we fix $r\leq \min\{k-1,l-2\}$. Fix $\vert\alpha\vert=r$. We then claim that $\partial_{v_i}\partial_v^\alpha \partial_{u_j}f=\partial_{u_j}\partial_{v_i}\partial_v^\alpha f$ for all $i\in\{1,\dots,n\}$. Indeed, by assumption
    \begin{equation}
        \partial_{v_i}\partial_v^\alpha\partial_{u_j}f=\partial_{v_i}\partial_{u_j}\partial_v^\alpha f
    \end{equation}
    and since both $\partial_{v_i}\partial_{u_j}\partial_v^\alpha f$ and $\partial_{u_j}\partial_{v_i}\partial_v^\alpha f$ exist continuously by assumption, the result in the case $\vert\alpha\vert=\vert\beta\vert=1$ implies that
    \begin{equation}
        \partial_{v_i}\partial_{u_j}\partial_v^\alpha f=\partial_{u_j}\partial_{v_i}\partial_v^\alpha f,
    \end{equation}
    giving us that 
    \begin{equation}
        \partial_v^\alpha\partial_{u_j}f=\partial_{u_j}\partial_v^\alpha f
    \end{equation}
    for $\vert\alpha\vert \leq r+1$. By induction we get that
    \begin{equation}
        \partial_v^\alpha\partial_{u_j}f=\partial_{u_j}\partial_v^\alpha f
    \end{equation}
    for all $\alpha$ with $\vert\alpha\vert\leq k$ and $\vert\alpha\vert+1\leq l$.

    We proceed by induction on $\vert\beta\vert$. Using the previous claim as a base case, suppose then that
    \begin{equation}
        \partial_v^\alpha\partial_u^\beta f=\partial_u^\beta\partial_v^\alpha f
    \end{equation}
    for all $\alpha$ and $\beta$ with $\vert\alpha\vert\leq k$ and $\vert \beta\vert\leq s$, where we fix $s\leq l-\vert\alpha\vert-1$. Fix $\vert\beta\vert=2$ and $\vert\alpha\vert\leq k$. We then claim that for any $j\in \{1,\dots,m\}$ we have that
    \begin{equation}
        \partial_v^\alpha\partial_{u_j}\partial_u^\beta f=\partial_{u_j}\partial_u^\beta\partial_v^\alpha f.
    \end{equation}
    Indeed, by assumption
    \begin{equation}
        \partial_{u_j}\partial_u^\beta\partial_v^\alpha f=\partial_{u_j}\partial_v^\alpha\partial_u^\beta f
    \end{equation}
    and by substituting $\partial_u^\beta f$ for $f$ in the result for $\vert\alpha\vert\leq k$ and $\vert\beta\vert=1$ we have that
    \begin{equation}
        \partial_{u_j}\partial_v^\alpha\partial_u^\beta f=\partial_v^\alpha\partial_{u_j}\partial_u^\beta f,
    \end{equation}
    giving us the claim. This finishes the proof by induction on $\vert\beta\vert$.

\subsection{Proof of Lemma~\ref{Anisotropicinverse}}

Define the map $F:V\times U\to\R^{m+n}$,
    \begin{equation}
        F:(v,u)\mapsto (v,f(v,u)).
    \end{equation}
    Then
    \begin{equation}
        JF(v,u)=\begin{bmatrix}
            \text{Id}_{m\times m} & 0 \\
            * & J_uf(v,u)
        \end{bmatrix},
    \end{equation}
    so that
    \begin{equation}
        \text{det}(JF)(v,u)=\text{det}(J_uf)(v,u)\neq 0
    \end{equation}
    for all $(v,u)\in V\times U$ by the fact that $f(v,\dummy)$ is a diffeomorphism for all $v\in V$. Since $F\in C^k(V\times U)$, we have by the inverse function theorem that $F^{-1}\in C^k(V\times f(U))$ and hence that $f^{-1}\in C^k(V\times f(U))$. 

    We continue by showing that $f^{-1}\in C_\ell(V\times f(U))$. The case $\ell=1$ is already established, so suppose $\ell=2$. Using the continuous differentiability of $f^{-1}$ lets us write
    \begin{equation}
        J_yf^{-1}(v,y)=(G\circ J_uf\circ F^{-1})(v,y),
    \end{equation}
    where $G:GL_n(\R)\to GL_n(\R)$ is the map 
    \begin{equation}
        G:A\mapsto A^{-1}.
    \end{equation}
    The map $G$ is $C^\infty$-smooth as a function of the matrix elements and hence since $F^{-1}$ is continuously differentiable with respect to $y$ and $J_uf$ is once continuously differentiable with respect to $u$, we get that $J_yf^{-1}\in C_{1}(V\times f(U))$ and hence that $f^{-1}\in C_{2}(V\times f(U))$. 
    
    Suppose then that $\ell\ge 3$. By the previous arguments we have that $f^{-1}\in C_2(V\times f(U))$. Suppose that $f^{-1}\in C_r(V\times f(U))$ for some $r$ with $2\leq r\leq \ell-1$. Then by the smoothness of $G$ and the fact that $J_uf\in C_{\ell-1}(V\times U)$ we get for all $\nu$ with $\vert\nu\vert=r-1$ that
    \begin{equation}
        \partial_y^{\nu}J_yf^{-1}=\partial_y^\nu(G\circ J_uf\circ F^{-1})
    \end{equation}
    is in $C_1(V\times U)$, showing that $J_yf\in C_r(V\times f(U))$ and thus that $f^{-1}\in C_{r+1}(V\times f(U))$. By induction we thus have that $f^{-1}\in C_\ell(V\times f(U))$.

    To establish that $f^{-1}\in C^k(V\times f(U))$, we begin by using the continuous differentiability of $f^{-1}$ to compute
    \begin{equation}
        \begin{split}
            0&=\partial_{v_i}y=\partial_{v_i}f(v,f^{-1}(v,y))\\
            &=(f_{v_i}\circ F^{-1})(v,y)+(J_uf\circ F^{-1})(v,y)\partial_{v_i}f^{-1}(v,y),
        \end{split}
    \end{equation}
    giving us the formula
    \begin{equation}
        \partial_{v_i}f^{-1}(v,y)=-(G\circ J_uf\circ F^{-1})(v,y)(f_{v_i}\circ F^{-1})(v,y).
    \end{equation}
    Here we see that for any fixed $s\leq k-1$ if $f^{-1}\in C^{s-1}(V\times f(U))$, then by the smoothness of $G$, the fact that $J_uf\in C^{\min\{k,\ell-1\}}(V\times U)\cap C_{\ell-1}(V\times U)$ and $f_{v_i}\in C^{k-1}(V\times U)\cap C_{\ell-1}(V\times U)$ we have that $\partial_{v_i}f^{-1}\in C^{s-1}(V\times f(U))$ and hence that $f^{-1}\in C^s(V\times f(U))$. By induction we get that $f^{-1}\in C^k(V\times f(U))$, which then lets us conclude that $f^{-1}\in C^k(V\times f(U))\cap C_\ell(V\times f(U))$.

    To verify that $\partial_y^\alpha f^{-1}\in C^{\min\{k,\ell-\vert\alpha\vert\}}(V\times f(U))$ for all $\alpha$ with $\vert\alpha\vert\leq \ell$, we note that the case $\ell=1$ is degenerate and hence start with the case $\ell=2$. We do this explicitly and treat the case $\ell\ge 3$ separately. Recall that
    \begin{equation}
        J_yf^{-1}=(G\circ J_uf\circ F^{-1}).
    \end{equation}
    Hence since $f^{-1}\in C^k(V\times f(U))$, by the smoothness of $G$ and the fact that $J_uf\in C^{1}(V\times U)\cap C_{1}(V\times U)$ we have that $J_yf^{-1}\in C^{1}(V\times f(U))$, thus giving us the case $\ell=2$.

    Let then $\ell\ge 3$. By the above we know that $\partial_y^\alpha f^{-1}\in C^{\min\{k,\ell-\vert\alpha\vert\}}(V\times f(U))$ for $\alpha$ with $\vert\alpha\vert\leq 2$. Suppose $\partial_y^\alpha f^{-1}\in C^{\min\{k,\ell-\vert\alpha\vert\}}(V\times f(U))$ for some $\alpha$ with $\vert\alpha\vert\leq \ell-1$. Let $\nu$ be such that $\vert\nu\vert=\vert\alpha\vert-1$ and let $i\in\{1,\dots,n\}$. Now
    \begin{equation}
        \partial_{y_i}\partial_y^\nu J_yf^{-1}=\partial_{y_i}\partial_y^\nu(G\circ J_u\circ F^{-1})
    \end{equation}
    and since $G$ is smooth, $\partial_u^\beta J_u\in C^{\min\{k,\ell-\vert\alpha\vert\}}(V\times U)\cap C_{\ell-\vert\alpha\vert}(V\times U)$ for $\beta$ with $\vert\beta\vert=\vert\alpha\vert-1$ and $F^{-1}\in C^{\min\{k,\ell-\vert\alpha\vert\}}(V\times f(U))$, we get that $\partial_{y_i}\partial_y^\nu J_yf^{-1}\in C^{\min\{k,\ell-1-\vert\nu\vert-1\}}(V\times f(U))=C^{\min\{k,\ell-\vert\alpha\vert-1\}}(V\times f(U))$. By induction on $\vert\alpha\vert$ we get that $\partial_y^\alpha f^{-1}\in C^{\min\{k,\ell-\vert\alpha\vert\}}(V\times f(U))$ for all $\alpha$ with $\vert\alpha\vert\leq \ell$.

    To verify that $\partial_v^\alpha f^{-1}\in C_{\ell-\vert\alpha\vert}(V\times f(U))$ for all $\alpha$ with $\vert \alpha\vert\leq k$, we again note that the case $\ell=1$ is degenerate and thus directly move on to the case $\ell\ge 2$. Let $i\in\{1,\dots,n\}$ and write
    \begin{equation}
        \partial_{v_i}f^{-1}=(G\circ J_uf\circ F^{-1})(f_{v_i}\circ F^{-1}).
    \end{equation}
    Since $G$ is smooth, $ J_uf\in C_{\ell-1}(V\times U)$ for $\beta$, $f_{v_i}\in C_{\ell-1}(V\times U)$ and $F^{-1}\in C_\ell(V\times f(U))$, we have that $\partial_{v_i}f^{-1}\in C_{\ell-1}(V\times f(U))$. If $\ell=2$ we are done. If $\ell\ge 3$, suppose $\partial_v^\alpha f^{-1}\in C_{\ell-\vert\alpha\vert}(V\times f(U))$ for all $\alpha$ with $\vert\alpha\vert\leq r\leq \ell-1$. Let $i\in \{1,\dots,n\}$. We can rewrite
    \begin{equation}
        \partial_{v_i}\partial_v^\alpha f^{-1}=\partial_{v_i}\partial_v^\nu\partial_{v_j}f^{-1}=-\partial_{v_i}\partial_v^\nu\left[(G\circ J_uf\circ F^{-1})(f_{v_i}\circ F^{-1})\right]
    \end{equation}
    for some $\nu$ with $\vert\nu\vert=\vert\alpha\vert-1$ and $j\in\{1,\dots,n\}$. Since $G$ is smooth, $\partial_{v_i}\partial_v^\nu J_u\in C_{\ell-\vert\nu\vert-1}(V\times U)=C_{\ell-\vert\alpha\vert-1}(V\times U)$, $\partial_u^\beta J_uf\in C_{\ell-\vert\alpha\vert-1}(V\times U)$ for $\beta$ with $\vert\beta\vert=\vert \alpha\vert$, $\partial_{v_i}\partial_v^\nu f_{v_i}\in C_{\ell-\vert\alpha\vert-1}(V\times U)$, $\partial_u^\beta f_{v_i}\in C_{\ell-\vert\alpha\vert-1}(V\times U)$ for $\beta$ with $\vert\beta\vert=\vert\alpha\vert$ and $\partial_{v_i}\partial_v^\nu F^{-1}\in C_{\ell-\vert\alpha\vert}(V\times f(U))$, we have that
    \begin{equation}
        \partial_{v_i}\partial_v^\alpha f^{-1}\in C_{\ell-\vert\alpha\vert-1}(V\times f(U)).
    \end{equation}
    By induction on $\vert\alpha\vert$ we obtain the claim.

\subsection{Proof of Lemma~\ref{anisotropicimplicit}}

Let us begin by showing that $\phi\in C_l(V\times U)$ for all $\ell$. Since $\phi$ is once continuously differentiable, we have by implicit differentiation that
    \begin{equation}
        \begin{split}
            \partial_{u_i}\phi(v,u)&=-(J_wF)^{-1}(v,u,\phi(v,u))F_{u_i}(v,u,\phi(v,u))\\
            &=-(G\circ J_wF)(v,u,\phi(v,u))F_{u_i}(v,u,\phi(v,u)).
        \end{split}
    \end{equation}

    Suppose that $\phi\in C_s(V\times U)$ for some $s\ge 1$. Then by the regularity of $G$, $J_wF$ and $F_{u_i}$ we see that $\partial_{u_i}\phi\in C_s(V\times U)$ and hence that $\phi\in C_{s+1}(V\times U)$. By induction we get that $\phi\in C_l(V\times U)$ for all $l$ and hence that $\phi\in C_\infty(\Tilde{V}\times\Tilde{U})$.

    Replacing $u$ by $v$ and $\ell$ by $k$ in the preceding argument lets us conclude that $\phi\in C^k(V\times U)$. To show the regularity of the $u$-derivatives of $\phi$, we proceed by proving explicitly the cases $l=1$ and $l=2$ and by induction the cases $l\ge 3$.
    
    Suppose $l=1$. Using the formula
    \begin{equation}
        \partial_{u_i}\phi(v,u)=-(G\circ J_wF)(v,u,\phi(v,u))F_{u_i}(v,u,\phi(v,u)),
    \end{equation}
    the regularity of $G,\ J_wF$ and  $F_{u_i}$ and the fact that $\phi\in C^k(V\times U)$, we have that $\partial_{u_i}\phi\in C^k(V\times U)$ for all $i$, giving us the claim. 

    Suppose $l=2$. Then using the case $l=1$, we have for all $i,j\in\{1,\dots,n\}$ that
    \begin{equation}
        \partial_{u_j}\partial_{u_i}\phi(v,u)=-\partial_{u_j}\left((G\circ J_wF)(v,u,\phi(v,u))F_{u_i}(v,u,\phi(v,u))\right)\in C^k(V\times U)
    \end{equation}
    by the regularity of $G,J_wF$ and $\partial_{u_i}F$.

    Suppose then that $l\ge 3$. We show by induction that $\partial_u^\alpha\partial_{u_i}\phi\in C^k(V\times U)$ for all $\alpha$ with $\vert\alpha\vert\leq l-1$ and all $i\in\{1,\dots,n\}$, which then lets us conclude that $\partial_u^\nu\phi\in C^k(V\times U)$ for all $\nu$ with $\vert\nu\vert\leq l$. 

    Using the case $l=2$ lets us obtain the base case for the induction. Suppose then that $\partial_u^\alpha\partial_{u_i}\phi\in C^k_v$ for all $\alpha$ with $\vert\alpha\vert\leq r\leq l-2$. Then
    \begin{equation}
        \partial_{u_j}\partial_u^\alpha\partial_{u_i}\phi=-\partial_{u_j}\partial_u^\alpha\left((G\circ J_wF)(v,u,\phi(v,u))F_{u_i}(v,u,\phi(v,u))\right)\in C^k(V\times U)
    \end{equation}
    by the fact that $\partial_u^\beta\phi\in C^k(\Tilde{V}\times\Tilde{U})$ for all $\beta$ with $\vert\beta\vert\leq r+1$ and the regularity of $G,\ J_wF$ and $F_{u_i}$, letting us conclude that 
    \begin{equation}
        \partial_u^\mu\phi\in C^k(V\times U)
    \end{equation}
    for all $\mu$ with $\vert\mu\vert\leq r+1$, giving us the induction step and thus letting us conclude the proof.

\section{Technical Lemmas for the Geodesic X-Ray Transform}\label{xraytechnicalappendix}

Here we show that the proofs of the technical Lemmas 3.1 and 3.2 in~\cite{IM2023} hold true unchanged under the regularity assumptions of Corollary~\ref{geodesicxray}.

\subsection{Regularity of the Geodesic Flow}

We recall that the geodesic equation is
\begin{equation}\label{geodesiceqn}
    \Ddot{\gamma}^i+2G^i(\gamma,\Dot{\gamma})=0,
\end{equation}
where the geodesic spray coefficients $G^i$ are defined in local coordinates as
\begin{equation}
    G^i(x,y):=\frac{1}{4}g^{i\ell}\left(2\frac{\partial g_{jk}}{\partial x^\ell}-\frac{\partial g_{j\ell}}{\partial x^k}\right)y^jy^\ell.
\end{equation}

\subsubsection{Existence and Uniqueness}

We see that if $c_{ijk\ell}\in C^k(M)$, then $G^i\in C^{k-1}_\infty(TM\setminus\{0\})$ and hence the existence and uniqueness of solutions to the geodesic equation is guaranteed when $c_{ijk\ell}\in C^2(M)$.

\subsubsection{Dependence on Initial Data}

By the standard ODE theory and the machinery developed so far, if $c_{ijk\ell}\in C^k(M)$ with $k\ge 2$, the solutions of the geodesic equation depend $C^{k-1}$-smoothly on the initial data. Especially if we restrict ourselves to geodesics with a lowest point, the geodesics depend $C^{k-1}$-smoothly on the lowest point.

\subsection{Taylor Expansions and Weak Reversibility}

Based on the regularity of geodesics discussed above, we check that the weak reversibility condition
\begin{equation}
    \begin{cases}
        \dddot{r}(r_0,0)=0, \\
        \Ddot{\theta}(r_0,0)=0,
    \end{cases}
\end{equation}
the change of coordinates $t\mapsto r$
and the expansions
\begin{equation}
    \begin{split}
        &\Ddot{r}(r_0,t)=a(r_0)+E_1(r_0,t),\\
        &\Dot{r}(r_0,t)=a(r_0)t+E_2(r_0,t)\\
        & r(r_0,t)-r_0=\frac{a(r_0)t^2}{2}+E_3(r_0,t)
    \end{split}
\end{equation}
where $a(r_0):=\Ddot{r}(r_0,0)$, $ E_1(r_0,t)=\mathcal{O}(t^2)$, $E_2(r_0,t)=\mathcal{O}(t^3)$ and $E_3(r_0,t)=\mathcal{O}(t^4)$ as $t\to 0$ continue to hold when $c\in C^3(M)$. We also need that the above estimates are uniform in $r_0$ in any closed annulus containing $r_0$. This follows by compactness and continuous dependence on initial data.

Beginning with the weak reversibility condition,
we see that we need the solutions of the geodesic equation to be three times differentiable in time. By standard ODE theory this holds when the spray coefficients $G^i$ are of class $C^1(TM\mO)$, which holds when $c\in C^2(M)$. 

Moving on to the estimates, assuming that $\partial_t^4r$ exists continuously, which holds when $c_{ijk\ell}\in C^3(M)$, we can write the Taylor expansion
\begin{equation} \label{Error1}
    \Ddot{r}(r_0,t)=a(r_0)+E_1(r_0,t),
\end{equation}
where $a(r_0):=\Ddot{r}(r_0,0)$ and
\begin{equation}
    E_1(r_0,t)=\mathcal{O}(t^2)
\end{equation}
as $t\to 0$. This follows by the weak reversibility condition and repeated application of the fundamental theorem of calculus:
\begin{equation}
    \begin{split}
        \Ddot{r}(r_0,t)=\Ddot{r}(r_0,0)+\int_0^t\dddot{r}(r_0,s)\ ds =\Ddot{r}(r_0,0)+\int_0^t\int_0^s\partial_h^4r(r_0,h)\ dhds.
    \end{split}
\end{equation}
Here
\begin{equation}
    E_1(r_0,t):=\int_0^t\int_0^s\partial_h^4r(r_0,h)\ dhds=\mathcal{O}(t^2)
\end{equation}
as $t\to 0$ by the estimate
\begin{equation}
    \left\vert\int_0^t\int_0^s\partial_h^4r(r_0,h)\ dhds\right\vert\leq \Vert \partial_t^4r(r_0,\dummy)\Vert_{C^0([0,T_{\partial M}])}t^2,
\end{equation}
where $T_{\partial M}$ is defined as the finite boundary hitting time for the geodesic and 
\begin{equation}
    \Vert \partial_t^4r(r_0,\dummy)\Vert_{C^0([0,T_{\partial M}])}<\infty
\end{equation}
by continuity and compactness.

Integrating \eqref{Error1} gives us
\begin{equation}
    \Dot{r}(r_0,t)=a(r_0)t+E_2(r_0,t)\quad \text{and}\quad r(r_0,t)-r_0=\frac{a(r_0)t^2}{2}+E_3(r_0,t),
\end{equation}
where $E_2(r_0,t)=\mathcal{O}(t^3)$ and $E_3(r_0,t)=\mathcal{O}(t^4)$ as $t\to 0$.

\bibliographystyle{abbrv}
\bibliography{ref.bib}

\end{document}